\documentclass[12pt,reqno]{amsart}
\usepackage{}
\usepackage{amsmath}
\usepackage{mathrsfs}
\usepackage{amssymb}
\usepackage{amsthm}
\usepackage{mathrsfs}
\usepackage[centertags]{amsmath}
\usepackage{amsfonts}
\usepackage[numbers,sort&compress]{natbib}
\usepackage{color}
\usepackage{extarrows}
\usepackage{fullpage}
\usepackage{amssymb}
\usepackage[colorlinks=true,  linkcolor=blue, citecolor=blue]{hyperref}
\input amssym.def
\input amssym.tex

\date{}

\newtheorem{Theorem}{Theorem}[section]

\newtheorem{Lemma}{Lemma}[section]
\newtheorem{Remark}{Remark}[section]

\newcommand\R{\mbox{\bf R}}

\newcommand\SR{\mbox{\scriptsize\bf R}}

\newcommand{\definition}{{\lower .5ex
  \hbox{$\>\>\stackrel{\triangle}{=}\>\>$} }}
\newcommand\supp{\mathop{\rm supp}}



\allowdisplaybreaks
\begin{document}

\baselineskip=22pt
\thispagestyle{empty}

\begin{center}
{\Large \bf  Spatial decay and nonlinear smoothing of the sixth-order Boussinesq equation}\\[1ex]

{\quad Xiangqian Yan\footnote{Email: yanxiangqian213@126.com}$^a$,\,Yongsheng Li\footnote{Email: yshli@scut.edu.cn}$^{a}$,\,Wei Yan\footnote{Email: 011133@htu.edu.cn}$^{b*}$}\\[1ex]

{$^a$School of Mathematics,
 South China University of Technology,}\\
 {Guangzhou, Guangdong 510640, China}\\[1ex]

{$^{b*}$College of Mathematics and Statistics, Henan Normal University,}\\
{Xinxiang, Henan 453007,   China}\\[1ex]

\end{center}

\bigskip
\bigskip

\noindent{\bf Abstract.}
In this paper, we study the initial value problem of the sixth-order Boussinesq equation with quadratic
and cubic nonlinearities in arbitrary spatial dimensions. First, by using the Fourier restriction norm method and a high-low frequency decomposition, we establish the nonlinear smoothing for this equation, namely, the integral form of the solution to the Duhamel formulation enjoys higher regularity than its linear counterpart. Finally, by using the nonlinear smoothing, we establish the uniform convergence of the integral term and  its spatial decay for each fixed $t$.


\medskip
\noindent {\bf Corresponding Author:} Wei Yan

\medskip
\noindent {\bf Email Address:} 011133@htu.edu.cn

\bigskip
\noindent {\bf MSC2020-Mathematics Subject Classification}: 35G25, 35L70, 42B37
\bigskip

\leftskip 0 true cm \rightskip 0 true cm

\newpage{}

\begin{center}
{\Large \bf  Spatial decay and nonlinear smoothing of the sixth-order Boussinesq equation}\\[1ex]

{\quad Xiangqian Yan\footnote{Email: yanxiangqian213@126.com}$^a$,\,Yongsheng Li\footnote{Email: yshli@scut.edu.cn}$^{a}$,\,Wei Yan\footnote{Email: 011133@htu.edu.cn}$^{b*}$}\\[1ex]

{$^a$School of Mathematics,
 South China University of Technology,}\\
 {Guangzhou, Guangdong 510640, China}\\[1ex]

{$^{b*}$College of Mathematics and Statistics, Henan Normal University,}\\
{Xinxiang, Henan 453007,   China}\\[1ex]

\end{center}

\noindent{\bf Abstract.}
In this paper, we study  the initial value problem of the sixth-order Boussinesq equation with quadratic
and cubic nonlinearities in arbitrary spatial dimensions. First, by using the Fourier restriction norm method and a high-low frequency decomposition,
we establish the nonlinear smoothing for this equation, namely, the integral form of the solution to the Duhamel formulation enjoys higher regularity than its linear counterpart. Finally, by using the nonlinear smoothing,
we establish the uniform convergence of the integral term and  its spatial decay for each fixed $t$.

\leftskip 0 true cm \rightskip 0 true cm

\newpage

\baselineskip=20pt

\bigskip
\bigskip

\tableofcontents

\section{Introduction}\label{sec1}

\setcounter{Theorem}{0} \setcounter{Lemma}{0}

\setcounter{section}{1}

\subsection{Background of the sixth-order Boussinesq equation}

\indent In this paper, we investigate the following  sixth-order Boussinesq equation
\begin{eqnarray}
&&u_{tt}-\Delta u+\beta\Delta^{2}u-\Delta^{3}u-\Delta(u^{k+1})=0,\,x\in\R^{n},\,n\in\mathbf{N}^{+},\, k=1,2.\label{1.01}\\
&&u(0,x)=f(x),u_{t}(0,x)=g_{x}(x),\label{1.02}
\end{eqnarray}
where $\beta=\pm 1$.

As a fundamental model for shallow-water waves and nonlinear atomic chains, the sixth-order Boussinesq equation was derived in \cite{CMV1996} to justify the classical fourth-order Boussinesq equation as a physical model. Subsequently, Daripa and Hu \cite{DH1999} and Maugin \cite{M1999} extended its use to nonlinear lattice dynamics in elastic crystals. The Cauchy problem for the sixth-order Boussinesq equation has since been extensively studied; see, e.g., \cite{EF2012, F2009, EW2014, EFW2012, WE2021,WL2020,GW2020,ET2025}. Following the approach in \cite{F2009}, Esfahani and Farah \cite{EF2012} established the local well-posedness of the sixth-order Boussinesq equation with quadratic nonlinearities for $(f,g)\in H^{s}(\R)\times H^{s-2}(\R)(s>-\frac{1}{2})$. Subsequently, Esfahani and Wang \cite{EW2014} improved this result to $s>-\frac{3}{4}$ by applying the method introduced in \cite{T2001}. Moreover, using mixed Lebesgue spaces and Strichartz estimates, Esfahani et al. \cite{EFW2012} established the local well-posedness of the sixth-order Boussinesq equation with general nonlinearity $|u|^{\alpha}u$ for $(f,g)\in L^{2}(\R)\times \dot{H}^{-2}(\R)(0<\alpha<4)$ and $(f,g)\in H^{1}(\R)\times \dot{H}^{-1}(\R)(\alpha>0)$.  Recently,  Wang and Li \cite{WL2020} considered the well-posedness and scattering of the higher-dimensional sixth-order Boussinesq equation with a general nonlinear term $f(u)=O(u^{k})$ for $k>1, k\in\mathbf{N}^{+}$, in the framework of modulation spaces. Very recently, Esfahani and Tesfahun \cite{ET2025} studied the well-posedness of the higher-dimensional sixth-order Boussinesq equation with quadratic and cubic nonlinearities in the framework of Bourgain spaces.

\subsection{Background of the nonlinear smoothing of dispersive equations}
Nonlinear smoothing, in the setting of nonlinear dispersive equations, refers to the phenomenon that the nonlinear form of the solution to the Duhamel formulation exhibits higher regularity than its linear counterpart. For further details on this topic, we refer the reader to \cite{BS2010,C2013,CLS2021,COS2024,C2021,CS2020,ET2013,ET2013-1,ET2013-2,YWY2025,YYY2026}. Correia et al. \cite{COS2024} proved the nonlinear smoothing for the higher-dimensional mZK, cubic/quintic NLS, and quartic KdV equations via the Fourier restriction norm method with frequency-restricted estimates.
Yan et al. \cite{YWY2025} employed maximal function estimates and Strichartz estimates to establish the nonlinear smoothing for the generalized Zakharov-Kuznetsov equation in dyadic mixed Lebesgue spaces. Nonlinear smoothing estimates have been applied in several directions. On the one hand, Compaan et al. \cite{CLS2021} used them to prove the uniform convergence for the Schr\"{o}dinger equation, showing that the $L_{x}^{\infty}$-norm  of the nonlinear part of the solution  in the Duhamel formulation tends to zero as $t\rightarrow0$. On the other hand, Yan et al. \cite{YYY2026} employed nonlinear smoothing to establish the pointwise spatial decay for the generalized Ostrovsky equation, that the nonlinear part of the solution in the Duhamel formulation decays to zero as $|x|\rightarrow\infty$.

\subsection{Motivation and main contents}
Recently, Esfahani and Tesfahun \cite{ET2025} established bilinear and trilinear estimates for the higher-dimensional sixth-order Boussinesq equation in Bourgain spaces. This naturally raises the following question: Can we apply the method in \cite{ET2025} to establish more general bilinear and trilinear estimates for proving the nonlinear smoothing of the higher-dimensional sixth-order Boussinesq equation?

Motivated by this question, in this paper, we study the nonlinear smoothing of the sixth-order Boussinesq equation with quadratic
and cubic nonlinearities in arbitrary spatial dimensions.
First, by combining the Fourier restriction norm method with a high-low frequency decomposition, we establish the nonlinear smoothing
for both the quadratic $(k=1)$ and cubic $(k=2)$ cases. More precisely, we have the following result:
Let $(f,g)\in H^{s}(\R^{n})\times H^{s-2}(\R^{n})$. Then, there exists $0<T<1$ such that for $t\in [-T,T]$,
\begin{eqnarray*}
&&u(t,x)=u_{1}(t,x)+u_{2}(t,x)+u_{3}(t,x)
\end{eqnarray*}
is the solution to (\ref{1.01})-(\ref{1.02}), where $u_{1}$ and $u_{2}$ are defined as in \eqref{1.05}
 and
\begin{eqnarray*}
&&u_{3}(t,x)=\int_{0}^{t}U_{2}(t-s)
\Delta(u^{k+1})ds
\end{eqnarray*}
satisfies
$u_{3}(t,x)\in  C([-T,T];H^{s+a}(\R^{n}))$.
Here, for the quadratic case ($k=1$), we require $s>s_{1}$, where $s_{1}$ and $a$ are defined as in \eqref{1.04} and \eqref{1.03}, respectively; for the cubic case ($k=2$), we require $s>s_{3}$, where $s_{3}$ and $a$ are defined as in \eqref{1.08} and \eqref{1.07}, respectively.
Finally, using this nonlinear smoothing, we establish the uniform convergence of the integral term
$u_{3}$ and  its spatial decay for each fixed $t$. Specifically, we obtain the following results.
\begin{itemize}
\item{\bf Quadratic case ($k=1$)}. Let $(f,g)\in H^{s}(\R^{n})\times H^{s-2}(\R^{n})$ with $s>s_{2}$, where $s_{2}$ is defined in \eqref{1.06}. Then
\begin{eqnarray*}
&&\lim\limits_{t\longrightarrow 0}\left\|u(t,x)-U_{1}(t)f(x)-U_{2}(t)g_{x}(x)\right\|_{L_{x}^{\infty}}=0,
\end{eqnarray*}
and
\begin{eqnarray*}
&&\lim\limits_{|x|\longrightarrow +\infty}(u(t,x)-U_{1}(t)f(x)-U_{2}(t)g_{x}(x))=0.
\end{eqnarray*}
\item{\bf Cubic case ($k=2$)}. Let $(f,g)\in H^{s}(\R^{n})\times H^{s-2}(\R^{n})$ with $s>s_{4}$, where $s_{4}$ is defined in \eqref{1.09}. Then
\begin{eqnarray*}
&&\lim\limits_{t\longrightarrow 0}\left\|u(t,x)-U_{1}(t)f(x)-U_{2}(t)g_{x}(x)\right\|_{L_{x}^{\infty}}=0,
\end{eqnarray*}
and
\begin{eqnarray*}
&&\lim\limits_{|x|\longrightarrow +\infty}(u(t,x)-U_{1}(t)f(x)-U_{2}(t)g_{x}(x))=0.
\end{eqnarray*}
\end{itemize}

\subsection{Introduction to notations and spaces}
Before stating the main results, we introduce the notation that will be used throughout the proofs.
$a\sim b$ means that there exists $C_{1}, C_{2}>0$ satisfying $C_{1}|a|\leq|b|\leq C_{2}|a|$. Let $\eta(t)$ be
a function in $C_{c}^{\infty}(\R)$ with $\eta(t)=1, t\in[0,1]$, $\eta(t)=0$, $t\geq2$. $N\in2^{\mathbf{Z}}$.
We define $ \langle x\rangle=(1+|x|^{2})^{\frac{1}{2}}$, $\phi_{\pm}(\xi)=\sqrt{|\xi|^{2}\pm|\xi|^{4}+|\xi|^{6}}$, and
\begin{eqnarray*}
&&\mathscr{F}_{x}f(\xi)=\frac{1}{(2\pi)^{\frac{n}{2}}}\int_{\SR^{n}}e^{-ix\cdot\xi}f(x)dx,\\ &&\mathscr{F}_{\xi}^{-1}f(x)=\frac{1}{(2\pi)^{\frac{n}{2}}}\int_{\SR^{n}}e^{ix\cdot\xi}\mathscr{F}_{x}f(\xi)d\xi,\\
&&\mathscr{F}_{t}f(\tau)=\frac{1}{(2\pi)^{\frac{1}{2}}}\int_{\SR}e^{-it\tau}f(t)dt,\\ &&\mathscr{F}_{\tau}^{-1}f(t)=\frac{1}{(2\pi)^{\frac{1}{2}}}\int_{\SR}e^{it\tau}\mathscr{F}_{t}f(\tau)d\tau,\\
&&\mathscr{F}_{xt}f(\xi,\tau)=\frac{1}{(2\pi)^{\frac{n+1}{2}}}\int_{\SR^{n+1}}e^{-it\tau}e^{-ix\cdot\xi}f(x,t)dxdt,\\ &&\mathscr{F}_{\xi\tau}^{-1}f(x,t)=\frac{1}{(2\pi)^{\frac{n+1}{2}}}\int_{\SR^{n+1}}e^{it\tau}e^{ix\cdot\xi}\mathscr{F}_{xt}f(\xi,\tau)d\xi d\tau,\\
&&P_{N}f(x)=\frac{1}{(2\pi)^{\frac{n}{2}}}\int_{\{\xi\in\SR^{n}:|\xi|\sim N\}}e^{ix\cdot\xi}\mathscr{F}_{x}f(\xi)d\xi,\\
&&U_{1}(t)f(x)=\frac{1}{(2\pi)^{\frac{n}{2}}}\int_{\SR^{n}}e^{ix\cdot\xi}
\frac{e^{it\phi_{\pm}(\xi)}+e^{-it\phi_{\pm}(\xi)}}{2}\mathscr{F}_{x}f(\xi)d\xi,\\
&&U_{2}(t)g_{x}(x)=\frac{1}{(2\pi)^{\frac{n}{2}}}
\int_{\SR^{n}}e^{ix\cdot\xi}
\frac{e^{it\phi_{\pm}(\xi)}-e^{-it\phi_{\pm}(\xi)}}{2i\phi_{\pm}(\xi)}\mathscr{F}_{x}g_{x}(\xi)d\xi.
\end{eqnarray*}
The space $X_{s,b}(\R^{n+1})$ is defined as follows
\begin{eqnarray*}
&&X_{s,b}(\R^{n+1})=\{u\in\mathcal{S}^{\prime}(\R^{n+1}):\|u\|_{X_{s,b}}=\left\|\langle|\tau|-\phi_{\pm}(\xi)\rangle^{b}
\langle\xi\rangle^{s}\mathscr{F}_{xt}
u(\tau,\xi)\right\|_{L_{\tau\xi}^{2}}<\infty\}.
\end{eqnarray*}

\subsection{Main results}

The main results of this paper are as follows:

\begin{Theorem} \label{Theorem1}(Nonlinear smoothing:  Quadratic case ($k=1$))
Let $n\geq 1$, $0<\epsilon\ll1$,
\begin{align}
\begin{cases}
0\leq a\leq 1-2\epsilon,&n=1,2,\\
0\leq a\leq \min\left\{1-2\epsilon, s+\frac{5-n}{2}-4\epsilon\right\},&n\geq 3,\label{1.03}
\end{cases}
\end{align}
and define
\begin{align}
s_{1}:=
\begin{cases}
-\frac{1}{8},&n=1,\\
\frac{n-3}{4},&2\leq n\leq 6,\\
\frac{n-5}{2},&n\geq 7.\label{1.04}
\end{cases}
\end{align}
Assume that $(f,g)\in H^{s}(\R^{n})\times H^{s-2}(\R^{n})$ with $s>s_{1}$. Then, there exists $0<T<1$ such that for $t\in [-T,T]$,
\begin{eqnarray*}
&&u(t,x)=u_{1}(t,x)+u_{2}(t,x)+u_{3}(t,x)
\end{eqnarray*}
is the solution to (\ref{1.01})-(\ref{1.02}), where
\begin{eqnarray}
&&u_{1}(t,x)+u_{2}(t,x)=U_{1}(t)f(x)+U_{2}(t)g_{x}(x)\in H^{s}(\R^{n}),\label{1.05}
\end{eqnarray}
 and
\begin{eqnarray*}
&&u_{3}(t,x)=\int_{0}^{t}U_{2}(t-s)
\Delta(u^{2})ds\in C([-T,T];H^{s+a}(\R^{n})).
\end{eqnarray*}
\end{Theorem}
\begin{Remark}
Whether the above result can be further improved is still under consideration.
\end{Remark}

\begin{Theorem} \label{Theorem2}(Uniform convergence of solutions of the nonlinear equation to those of the linear equation:  Quadratic case ($k=1$))
Let $(f,g)\in H^{s}(\R^{n})\times H^{s-2}(\R^{n})$  with $s>s_{2}$. Then, we have
\begin{eqnarray*}
&&\lim\limits_{t\longrightarrow 0}\left\|u(t,x)-U_{1}(t)f(x)-U_{2}(t)g_{x}(x)\right\|_{L_{x}^{\infty}}=0,
\end{eqnarray*}
where
\begin{align}
s_{2}:=
\begin{cases}
-\frac{1}{8},&n=1,\\
\frac{n-3}{4},&n=2,\\
\frac{2n-5}{4},&n\geq 3.\label{1.06}
\end{cases}
\end{align}
\end{Theorem}
\begin{Remark}
Theorem 1.2 is an application of Theorem 1.1. We believe that the above result is not optimal, especially in the case $n=1$. By employing the method in \cite{EF2012}, it is possible to improve the result to
$s>-\frac{1}{4}$, for $n=1$. Whether the result can be further improved in other cases remains under investigation.
\end{Remark}

\begin{Theorem} \label{Theorem3}(Spatial decay for each fixed $t$: Quadratic case ($k=1$)).
Let $(f,g)\in H^{s}(\R^{n})\times H^{s-2}(\R^{n})$  with $s>s_{2}$. Then, for $t\in [-T,T],$  we have
\begin{eqnarray*}
&&\lim\limits_{|x|\longrightarrow +\infty}(u(t,x)-U_{1}(t)f(x)-U_{2}(t)g_{x}(x))=0.
\end{eqnarray*}

\end{Theorem}
\begin{Remark}
Theorem 1.3 is an application of Theorem 1.1.
\end{Remark}

\begin{Theorem} \label{Theorem4}(Nonlinear smoothing: Cubic case ($k=2$))
Let $n\geq 1$, $0<\epsilon\ll1$,
\begin{align}
\begin{cases}
0\leq a\leq 1-2\epsilon,&n=1\\
0\leq a\leq \min\left\{1-2\epsilon, 2s+\frac{5-2n}{2}-4\epsilon\right\},&n\geq 2,\label{1.07}
\end{cases}
\end{align}
and define
\begin{align}
s_{3}:=
\begin{cases}
-\frac{1}{6},&n=1,\\
\frac{2n-3}{6},&2\leq n\leq 4,\\
\frac{2n-5}{4},&n\geq5.\label{1.08}
\end{cases}
\end{align}
Assume that $(f,g)\in H^{s}(\R^{n})\times H^{s-2}(\R^{n})$  with $s>s_{3}$. Then, there exists $0<T<1$ such that for $t\in [-T,T]$,
\begin{eqnarray*}
&&u(t,x)=u_{1}(t,x)+u_{2}(t,x)+u_{3}(t,x)
\end{eqnarray*}
is the solution to (\ref{1.01})-(\ref{1.02}), where
\begin{eqnarray*}
&&u_{1}(t,x)+u_{2}(t,x)=U_{1}(t)f(x)+U_{2}(t)g_{x}(x)\in H^{s}(\R^{n}),
\end{eqnarray*}
 and
\begin{eqnarray*}
&&u_{3}(t,x)=\int_{0}^{t}U_{2}(t-s)
\Delta(u^{3})ds\in C([-T,T];H^{s+a}(\R^{n})).
\end{eqnarray*}
\end{Theorem}
\begin{Remark}
Whether the above result can be further improved is still under consideration.
\end{Remark}

\begin{Theorem} \label{Theorem5}(Uniform convergence of solutions of the nonlinear equation to those of the linear equation: Cubic case ($k=2$))
Let $(f,g)\in H^{s}(\R^{n})\times H^{s-2}(\R^{n})$  with $s>s_{4}$. Then, we have
\begin{eqnarray*}
&&\lim\limits_{t\longrightarrow 0}\left\|u(t,x)-U_{1}(t)f(x)-U_{2}(t)g_{x}(x)\right\|_{L_{x}^{\infty}}=0,
\end{eqnarray*}
where
\begin{align}
s_{4}:=
\begin{cases}
-\frac{1}{6},&n=1,\\
\frac{3n-5}{6},&n\geq2.\label{1.09}
\end{cases}
\end{align}

\end{Theorem}

\begin{Remark}
Theorem 1.5 is an application of Theorem 1.4.
\end{Remark}

\begin{Theorem} \label{Theorem6}(Spatial decay for each fixed $t$: Cubic case ($k=2$)).
Let $(f,g)\in H^{s}(\R^{n})\times H^{s-2}(\R^{n})$ with $s>s_{4}$. Then, for $t\in [-T,T],$  we have
\begin{eqnarray*}
&&\lim\limits_{|x|\longrightarrow +\infty}(u(t,x)-U_{1}(t)f(x)-U_{2}(t)g_{x}(x))=0.
\end{eqnarray*}

\end{Theorem}

\begin{Remark}
Theorem 1.6 is an application of Theorem 1.4.
\end{Remark}

The rest of the paper is arranged as follows.
In Section 2,  we give some preliminaries.
In Section 3, we prove the bilinear and trilinear estimates. In Section 4, we prove Theorem 1.1.
In Section 5, we prove Theorems 1.2-1.3.
In Section 6, we prove Theorems 1.4-1.6.

\bigskip
\section{Preliminaries}\label{sec2}

\setcounter{equation}{0}

\setcounter{Theorem}{0}

\setcounter{Lemma}{0}

\setcounter{section}{2}

In this section, we present some preliminary lemmas that will be used in the subsequent proofs.

\begin{Lemma}\label{Lemma2.1}
Let $n\geq 1$, $0<T<1$, $0<\epsilon\ll1$, $\frac{1}{2}<b<1$, $N_{1}\geq N_{2}$, $N_{1}\gg1$, and $\supp\mathscr{F}_{x} u_{i}\subset\{\xi\in\R^{n}:|\xi|\sim N_{i}\},\,i=1,2$. Then, we have
\begin{eqnarray}
&&\left\|u_{1}u_{2}\right\|_{L_{xT}^{2}}\leq CN_{1}^{-\frac{1}{4}}\|u_{1}\|_{X_{0,b}}\|u_{2}\|_{X_{0,b}},\, n=1,\label{2.01}\\
&&\left\|u_{1}u_{2}\right\|_{L_{xT}^{2}}\leq CN_{2}^{\frac{n}{2}-1+2\epsilon}
N_{1}^{-\frac{1}{2}+\epsilon}\|u_{1}\|_{X_{0,b}}\|u_{2}\|_{X_{0,b}},\, n\geq2.\label{2.02}
\end{eqnarray}

\end{Lemma}

For Lemma 2.1, we refer to \cite[Lemma 9]{ET2025}.

\begin{Lemma}\label{Lemma2.2}
Let $n\geq 1$, $0<T<1$, $0<\epsilon\ll1$, $\frac{1}{2}<b<1$, $N_{1}\geq N_{2}\geq N_{3}$, and $\supp\mathscr{F}_{x} u_{i}\subset\{\xi\in\R^{n}:|\xi|\sim N_{i}\},\,i=1,2,3$. Then:

\noindent for $n=1$, we have
\begin{eqnarray}
&&\left\|u_{1}u_{2}u_{3}\right\|_{L_{xT}^{2}}\leq CN_{1}^{-\frac{1}{4}}N_{3}^{\frac{1}{2}}\prod_{i=1}^{3}\|u_{i}\|_{X_{0,b}}
,\, N_{1}\gg 1,\label{2.03}\\
&&\left\|u_{1}u_{2}u_{3}\right\|_{L_{xT}^{2}}\leq CN_{2}^{-\frac{1}{4}}N_{1}^{-\frac{1}{4}}\prod_{i=1}^{3}\|u_{i}\|_{X_{0,b}}
,\, N_{2}\gg 1,\label{2.04}\\
&&\left\|u_{1}u_{2}u_{3}\right\|_{L_{xT}^{2}}\leq CN_{3}^{\frac{1}{3}}N_{1}^{-\frac{1}{6}}N_{2}^{-\frac{1}{6}}\prod_{i=1}^{3}\|u_{i}\|_{X_{0,b}}
,\, N_{2}\gg1,\label{2.05}
\end{eqnarray}

\noindent for $n\geq2$, $N_{1}\gg1$, we have
\begin{eqnarray}
&&\left\|u_{1}u_{2}u_{3}\right\|_{L_{xT}^{2}}\leq CN_{3}^{\frac{n}{2}}N_{2}^{\frac{n}{2}-1+2\epsilon}
N_{1}^{-\frac{1}{2}+\epsilon}\prod_{i=1}^{3}\|u_{i}\|_{X_{0,b}}.\label{2.06}
\end{eqnarray}

\end{Lemma}

\noindent{\bf Proof.}
For \eqref{2.03}-\eqref{2.04} and \eqref{2.06}, we refer to \cite[Lemma 10]{ET2025}. By using \cite[Lemma 5]{ET2025}, H\"{o}lder inequality and $\dot{H}^{\frac{1}{3}}(\R)\hookrightarrow L^{6}(\R)$, we have
\begin{eqnarray}
&&\left\|u_{1}u_{2}u_{3}\right\|_{L_{xT}^{2}}\leq CT^{\frac{1}{6}}\|u_{1}\|_{L_{xT}^{6}}\|u_{2}\|_{L_{xT}^{6}}\|u_{3}\|_{L_{T}^{\infty}L_{x}^{6}}\nonumber\\
&&\leq CT^{\frac{1}{6}}N_{3}^{\frac{1}{3}}\|u_{1}\|_{L_{xT}^{6}}\|u_{2}\|_{L_{xT}^{6}}\|u_{3}\|_{L_{T}^{\infty}L_{x}^{2}}\nonumber\\
&&\leq CN_{3}^{\frac{1}{3}}N_{1}^{-\frac{1}{6}}N_{2}^{-\frac{1}{6}}\prod_{i=1}^{3}\|u_{i}\|_{X_{0,b}}.\label{2.07}
\end{eqnarray}
From \eqref{2.07},  we have that \eqref{2.05} is valid.

This completes the proof of Lemma 2.2.

\begin{Lemma}\label{Lemma2.3}
Let $n\geq 1$, $s\in\R$, $0<T<1$, $-\frac{1}{2}<b^{\prime}\leq 0\leq b\leq b^{\prime}+1$ and $(f,g)\in H^{s}(\R^{n})\times H^{s-2}(\R^{n})$. Then, we have
\begin{eqnarray}
&&\left\|\eta(t)\left(U_{1}(t)f+U_{2}(t)g_{x}\right)\right\|_{X_{s,b}}\leq C\left(\left\|f\right\|_{H^{s}}+\left\|g\right\|_{H^{s-2}}\right),\label{2.08}\\
&&\left\|\eta\left(\frac{t}{T}\right)\int_{0}^{t}U_{2}(t-s)f(u)(s)ds\right\|_{X_{s,b}}\nonumber\\
&&\leq CT^{1-(b-b^{\prime})}\left\|\mathscr{F}_{xt}^{-1}\left(\frac{\mathscr{F}_{xt}f(u)(\tau,\xi)}{2i\phi_{\pm}(\xi)}\right)\right\|_{X_{s,b^{\prime}}}.\label{2.09}
\end{eqnarray}

\end{Lemma}

For Lemma 2.3, we refer to \cite[Lemmas 2.1, 2.2]{EF2012}.

\begin{Lemma}\label{Lemma2.4}
Let $n\geq 1$, $s\in\R$, $b>\frac{1}{2}$. Then, we have
\begin{eqnarray}
&&X_{s,b}(\R^{n+1})\subset C(\R; H^{s}(\R^{n})).\label{2.010}
\end{eqnarray}

\end{Lemma}

\noindent {\bf Proof.}
By using a proof similar to that of \cite[Lemma 3]{F2009}, we obtain
\begin{eqnarray}
&&\|u\|_{H^{s}}\leq C\|u\|_{X_{s,b}}.\label{2.010}
\end{eqnarray}
Next, we prove the continuity of $\|u\|_{H^{s}}$ with respect to $t$. Noting that
\begin{eqnarray}
&&u=u_{1}+u_{2},\label{2.011}
\end{eqnarray}
where
\begin{eqnarray*}
&&u_{1}=\mathscr{F}_{t}^{-1}(I_{\{\tau\in\SR:\tau\leq 0\}}\mathscr{F}_{t}u(x,\tau)),\,u_{2}=\mathscr{F}_{t}^{-1}(I_{\{\tau\in\SR:\tau> 0\}}\mathscr{F}_{t}u(x,\tau)).
\end{eqnarray*}
Obviously, we have
\begin{eqnarray}
\|u_{1}\|_{X_{s,b}}=\|U(t)u_{1}\|_{H_{x}^{s}H_{t}^{b}}\leq C\|u\|_{X_{s,b}},\,\|u_{2}\|_{X_{s,b}}=\|U(-t)u_{2}\|_{H_{x}^{s}H_{t}^{b}}\leq C\|u\|_{X_{s,b}},\label{2.012}
\end{eqnarray}
where
\begin{eqnarray*}
&&U(t)f=\int_{\SR^{n}}e^{ix\cdot\xi}e^{it\phi_{\pm}(\xi)}\mathscr{F}_{x}f(\xi)d\xi.
\end{eqnarray*}
By using \eqref{2.011}, in order to prove the continuity of $\|u\|_{H^{s}}$ with respect to $t$, it suffices to prove the continuity of
$\|u_{1}\|_{H^{s}}$ and $\|u_{2}\|_{H^{s}}$ in $t$. Moreover, since the proof of the continuity of $\|u_{2}\|_{H^{s}}$ in $t$ is similar to that of
$\|u_{1}\|_{H^{s}}$, we only need to prove the continuity of $\|u_{1}\|_{H^{s}}$ with respect to $t$.

\noindent We denote
\begin{eqnarray*}
&&v(t):=U(t)u_{1}.
\end{eqnarray*}
Then it suffices to prove that for any
$t_{1}, t_{2}\in\R$ and $\epsilon>0$, there exists $\delta(\epsilon)>0$ such that whenever $|t_{1}-t_{2}|<\delta(\epsilon)$, we have
\begin{eqnarray}
&&\Big|\left\|U(-t_{1})v(t_{1})\right\|_{H^{s}}-\left\|U(-t_{2})v(t_{2})\right\|_{H^{s}}\Big|<\epsilon.\label{2.013}
\end{eqnarray}
Now, we prove \eqref{2.013}. From \eqref{2.012}, by using the H\"older inequality, for $M\geq 2026$(fixed),  we have that
\begin{eqnarray}
&&\int_{|\tau|\geq M}\left(\int_{\SR^{n}}\langle\xi\rangle^{2s}|\mathscr{F}_{xt}v(\tau,\xi)|^{2}d\xi\right)^{1/2}d\tau\nonumber\\
&&\leq \left[\int_{|\tau|\geq M}\langle\tau\rangle^{2b}\int_{\SR^{n}}\langle\xi\rangle^{2s}|\mathscr{F}_{xt}v(\tau,\xi)|^{2}d\xi d\tau\right]^{1/2}
\times \left[\int_{|\tau|\geq M}\langle\tau\rangle^{-2b}d\tau\right]^{1/2}\nonumber\\
&&\leq C\epsilon.\label{2.014}
\end{eqnarray}
For $M\geq 2026$(fixed), by using Minkowski's inequality, we have
\begin{eqnarray}
&&\Big|\|U(-t_{1})v(t_{1})\|_{H^{s}}-\|U(-t_{2})v(t_{2})\|_{H^{s}}\Big|=\Big|\|v(t_{1})\|_{H^{s}}-\|v(t_{2})\|_{H^{s}}\Big|\nonumber\\
&&\leq \|v(t_{1})-v(t_{2})\|_{H^{s}}=\left\|\int_{\SR}
\left(e^{it_{1}\tau}-e^{it_{2}\tau}\right)
\mathscr{F}_{t}v(x,\tau)d\tau\right\|_{H^{s}}\nonumber\\
&&=\left(\int_{\SR^{n}}\langle\xi\rangle^{2s}\left|\int_{\SR}
\left(e^{it_{1}\tau}-e^{it_{2}\tau}\right)
\mathscr{F}_{xt}v(\xi,\tau)d\tau\right|^{2}d\xi\right)^{\frac{1}{2}}\nonumber\\
&&\leq \left(\int_{\SR^{n}}\langle\xi\rangle^{2s}\left(\int_{\SR}
\left|\left(e^{it_{1}\tau}-e^{it_{2}\tau}\right)
\mathscr{F}_{xt}v(\xi,\tau)\right|d\tau\right)^{2}d\xi\right)^{\frac{1}{2}}\nonumber\\
&&\leq \int_{\SR}\left|e^{it_{1}\tau}-e^{it_{2}\tau}\right|
\left(\int_{\SR^{n}}\langle\xi\rangle^{2s}
\left|\mathscr{F}_{xt}v(\xi,\tau)\right|^{2}d\xi\right)^{\frac{1}{2}} d\tau\nonumber\\
&&=\int_{|\tau|\leq M}\left|e^{it_{1}\tau}-e^{it_{2}\tau}\right|
\left(\int_{\SR^{n}}\langle\xi\rangle^{2s}
\left|\mathscr{F}_{xt}v(\xi,\tau)\right|^{2}d\xi\right)^{\frac{1}{2}} d\tau\nonumber\\
&&\quad+\int_{|\tau|\geq M}\left|e^{it_{1}\tau}-e^{it_{2}\tau}\right|
\left(\int_{\SR^{n}}\langle\xi\rangle^{2s}
\left|\mathscr{F}_{xt}v(\xi,\tau)\right|^{2}d\xi\right)^{\frac{1}{2}} d\tau\nonumber\\
&&=I_{1}+I_{2},\label{2.015}
\end{eqnarray}
where
\begin{eqnarray*}
&&I_{1}=\int_{|\tau|\leq M}\left|e^{it_{1}\tau}-e^{it_{2}\tau}\right|
\left(\int_{\SR^{n}}\langle\xi\rangle^{2s}
\left|\mathscr{F}_{xt}v(\xi,\tau)\right|^{2}d\xi\right)^{\frac{1}{2}} d\tau,\\
&&I_{2}=\int_{|\tau|\geq M}\left|e^{it_{1}\tau}-e^{it_{2}\tau}\right|
\left(\int_{\SR^{n}}\langle\xi\rangle^{2s}
\left|\mathscr{F}_{xt}v(\xi,\tau)\right|^{2}d\xi\right)^{\frac{1}{2}} d\tau.
\end{eqnarray*}
For $I_{1}$, by using \eqref{2.012}, we have
\begin{eqnarray}
&&I_{1}=\int_{|\tau|\leq M}\left|e^{it_{1}\tau}-e^{it_{2}\tau}\right|\left(\int_{\SR^{n}}\langle\xi\rangle^{2s}
\left|\mathscr{F}_{xt}v(\xi,\tau)\right|^{2}d\xi\right)^{\frac{1}{2}} d\tau\nonumber\\
&&\leq M|t_{1}-t_{2}|\int_{|\tau|\leq M}\left(\int_{\SR^{n}}\langle\xi\rangle^{2s}
\left|\mathscr{F}_{xt}v(\xi,\tau)\right|^{2}d\xi\right)^{\frac{1}{2}} d\tau\nonumber\\
&&\leq M|t_{1}-t_{2}|\left\|J_{x}^{s}J_{t}^{b}v\right\|_{L_{xt}^{2}}\nonumber\\
&&\leq CM|t_{1}-t_{2}|.\label{2.016}
\end{eqnarray}
For $I_{2}$, from \eqref{2.014}, we have
\begin{eqnarray}
&&I_{2}=\int_{|\tau|\geq M}\left|e^{it_{1}\tau}-e^{it_{2}\tau}\right|\left(\int_{\SR^{n}}\langle\xi\rangle^{2s}
\left|\mathscr{F}_{xt}v(\xi,\tau)\right|^{2}d\xi\right)^{\frac{1}{2}} d\tau.\nonumber\\
&&\leq 4\int_{|\tau|\geq M}\left(\int_{\SR^{n}}\langle\xi\rangle^{2s}
\left|\mathscr{F}_{xt}v(\xi,\tau)\right|^{2}d\xi\right)^{\frac{1}{2}} d\tau
\leq C\epsilon.\label{2.017}
\end{eqnarray}
From \eqref{2.016}-\eqref{2.017}, we have
\begin{eqnarray}
&&I_{2}\leq C(M|t_{1}-t_{2}|+\epsilon).\label{2.018}
\end{eqnarray}
When  $|t_{1}-t_{2}|<\delta<\frac{\epsilon}{M}$, then, we have that \eqref{2.013} is valid.

This completes the proof of Lemma 2.4.

\bigskip
\section{Bilinear and trilinear estimates}\label{sec3}

\setcounter{equation}{0}

\setcounter{Theorem}{0}

\setcounter{Lemma}{0}

\setcounter{section}{3}

In this section, we establish some bilinear and trilinear estimates that play a crucial role in the proofs of Theorems 1.1-1.6.

\begin{Lemma}\label{Lemma3.1}
Let $n=1$, $0<\epsilon\ll1$, $0<T<1$, $s\geq-\frac{1}{8}+\frac{\epsilon}{2}$, $b_{0}=-\frac{1}{2}+\frac{\epsilon}{12}$, $b=\frac{1}{2}+\frac{\epsilon}{24}$ and
\begin{eqnarray*}
&&0\leq a\leq \min\left\{1-2\epsilon, s+\frac{5}{4}-\epsilon\right\}=1-2\epsilon.
\end{eqnarray*}
Then, we have
\begin{eqnarray}
&&\left\|\eta\left(\frac{t}{T}\right)\mathscr{F}_{xt}^{-1}\left(\frac{|\xi|^{2}\mathscr{F}_{xt}(u^{2})}{2i\phi_{\pm}(\xi)}\right)\right\|_{X_{s+a,b_{0}}}\leq CT^{-b_{0}}\|u\|_{X_{s,b}}^{2}.\label{3.01}
\end{eqnarray}
\end{Lemma}
\noindent {\bf Proof.}
Inspired by \cite{ET2025}, we prove Lemma 3.1. Noting that
\begin{eqnarray}
&&|\phi_{\pm}(\xi)|\sim |\xi|\langle \xi\rangle^{2}.\label{3.02}
\end{eqnarray}
From \eqref{3.02}, to prove \eqref{3.01}, it suffices to show that
\begin{eqnarray}
&&\left\|\eta\left(\frac{t}{T}\right)|D|\langle D\rangle^{-2}(u^{2})\right\|_{X_{s+a,b_{0}}}\leq CT^{-b_{0}}\|u\|_{X_{s,b}}^{2}.\label{3.03}
\end{eqnarray}
By using \cite[Lemma 2.11]{T2006}, we have
\begin{eqnarray}
&&\left\|\eta\left(\frac{t}{T}\right)|D|\langle D\rangle^{-2}(u^{2})\right\|_{X_{s+a,b_{0}}}\leq CT^{-b_{0}}
\left\|\eta\left(\frac{t}{T}\right)|D|\langle D\rangle^{-2}(u^{2})\right\|_{X_{s+a,0}}\nonumber\\
&&=CT^{-b_{0}}\left\||D|\langle D\rangle^{-2}(u^{2})\right\|_{L_{T}^{2}H_{x}^{s+a}}.\label{3.04}
\end{eqnarray}
By using \eqref{3.04}, to prove \eqref{3.01}, it suffices to show that
\begin{eqnarray}
&&\left\||D|\langle D\rangle^{s+a-2}(u^{2})\right\|_{L_{xT}^{2}}\leq C\|u\|_{X_{s,b}}^{2}.\label{3.05}
\end{eqnarray}
According to the duality idea, to prove \eqref{3.05}, we only need to prove
\begin{eqnarray}
&&\left|\int_{0}^{T}\int_{\SR}|D|\langle D\rangle^{s+a-2}\left(\langle D\rangle^{-s} u_{1}\langle D\rangle^{-s} u_{2}\right)(x,t)\overline{h}(x,t)dxdt\right|\nonumber\\
&&\leq C\|h\|_{L_{xT}^{2}}\prod_{j=1}^{2}\|u_{j}\|_{X_{0,b}}.\label{3.06}
\end{eqnarray}
We denote
\begin{eqnarray}
&&I_{1}:=\int_{0}^{T}\int_{\SR}|D|\langle D\rangle^{s+a-2}\left(\langle D\rangle^{-s} u_{1}\langle D\rangle^{-s} u_{2}\right)(x,t)\overline{h}(x,t)dxdt.\label{3.07}
\end{eqnarray}
Let $P_{N}h$ and $P_{N_{j}}u_{j}(j=1,2)$ be the dyadic decompositions of $h$ and $u_{j}$, respectively.
Then, we have that $\supp \mathscr{F}_{x}P_{N}h\subset\{\xi\in\R:|\xi|\sim N\}$, and $\supp \mathscr{F}_{x}P_{N_{j}}u_{j}\subset \{\xi\in\R:|\xi|\sim N_{j}\}$, where $N$ and $N_{j}$ are dyadic numbers. We define $\sum=\sum\limits_{N_{1}, N_{2}, N}$. Without loss of generality, we assume $N_{1}\geq N_{2}$.

\noindent We consider the following cases: {\bf Case 1}: $N_{1}\leq 10$; {\bf Case
2}: $N_{1}\geq 10$, $N_{1}\gg N_{2}$ and {\bf Case 3}: $N_{1}\geq 10$, $N_{1}\sim N_{2}\gg N$; {\bf Case 4}: $N_{1}\geq 10$, $N_{1}\sim N_{2}\sim N$.

\noindent {\bf Case 1}: When $N_{1}\leq 10$, by using H\"{o}lder inequality and Bernstein inequality, we have
\begin{eqnarray}
&&|I_{1}|\leq C\sum N\langle N\rangle^{s+a-2}\langle N_{1}\rangle^{-s}\langle N_{2}\rangle^{-s}
\|P_{N_{1}}u_{1}\|_{L_{xT}^{2}}\|P_{N_{2}}u_{2}\|_{L_{xT}^{\infty}}\|P_{N}h\|_{L_{xT}^{2}}\nonumber\\
&&\leq C\sum NN_{2}^{\frac{1}{2}}\langle N_{1}\rangle^{-s}\langle N_{2}\rangle^{-s}
\|P_{N_{1}}u_{1}\|_{X_{0,b}}\|P_{N_{2}}u_{2}\|_{X_{0,b}}\|P_{N}h\|_{L_{xT}^{2}}\nonumber\\
&&\leq C\sum N N_{1}^{\frac{1}{4}}N_{2}^{\frac{1}{4}}\|P_{N_{1}}u_{1}\|_{X_{0,b}}\|P_{N_{2}}u_{2}\|_{X_{0,b}}\|P_{N}h\|_{L_{xT}^{2}}
\leq C\|h\|_{L_{xT}^{2}}\prod_{j=1}^{2}\|u_{j}\|_{X_{0,b}}.\label{3.08}
\end{eqnarray}
{\bf Case 2}: When $N_{1}\geq 10$,  $N_{1}\gg N_{2}$, we consider the following cases: {\bf Case 2A}: $N_{2}\leq 8$ and {\bf Case
2B}: $N_{2}\geq 8$.

\noindent {\bf Case 2A}: $N_{1}\geq 10$,  $N_{1}\gg N_{2}$, $N_{2}\leq 8$, we have $N_{1}\sim N$.
By using H\"{o}lder inequality and Bernstein inequality, $0\leq a\leq 1-2\epsilon$, we have
\begin{eqnarray}
&&|I_{1}|\leq C\sum NN^{s+a-2}N_{1}^{-s}\langle N_{2}\rangle^{-s}
\|P_{N_{1}}u_{1}\|_{L_{xT}^{2}}\|P_{N_{2}}u_{2}\|_{L_{xT}^{\infty}}\|P_{N}h\|_{L_{xT}^{2}}\nonumber\\
&&\leq C\sum N^{a-1}N_{2}^{\frac{1}{2}}\langle N_{2}\rangle^{-s}
\|P_{N_{1}}u_{1}\|_{X_{0,b}}\|P_{N_{2}}u_{2}\|_{X_{0,b}}\|P_{N}h\|_{L_{xT}^{2}}\nonumber\\
&&\leq C\sum N^{a-1}N_{2}^{\frac{1}{2}}\|P_{N_{1}}u_{1}\|_{X_{0,b}}\|P_{N_{2}}u_{2}\|_{X_{0,b}}\|P_{N}h\|_{L_{xT}^{2}}
\leq C\|h\|_{L_{xT}^{2}}\prod_{j=1}^{2}\|u_{j}\|_{X_{0,b}}.\label{3.09}
\end{eqnarray}
{\bf Case 2B}: $N_{1}\geq 10$,  $N_{1}\gg N_{2}$, $N_{2}\geq 8$, we have $N_{1}\sim N$.
By using H\"{o}lder inequality and \eqref{2.01}, $s\geq-\frac{1}{8}+\frac{\epsilon}{2}$, $0\leq a\leq \min\{1-2\epsilon, s+\frac{5}{4}-\epsilon\}$, we have
\begin{eqnarray}
&&|I_{1}|\leq C\sum NN^{s+a-2}N_{1}^{-s}N_{2}^{-s}
\|P_{N_{1}}u_{1}P_{N_{2}}u_{2}\|_{L_{xT}^{2}}\|P_{N}h\|_{L_{xT}^{2}}\nonumber\\
&&\leq C\sum N_{1}^{-\epsilon}N_{1}^{a-1+\epsilon}N_{1}^{-\frac{1}{4}}N_{2}^{-s}
\|P_{N_{1}}u_{1}\|_{X_{0,b}}\|P_{N_{2}}u_{2}\|_{X_{0,b}}\|P_{N}h\|_{L_{xT}^{2}}\nonumber\\
&&\leq C\sum N_{1}^{-\epsilon} N_{2}^{-s+a-\frac{5}{4}+\epsilon}\|P_{N_{1}}u_{1}\|_{X_{0,b}}\|P_{N_{2}}u_{2}\|_{X_{0,b}}\|P_{N}h\|_{L_{xT}^{2}}\nonumber\\
&&\leq C\|h\|_{L_{xT}^{2}}\prod_{j=1}^{2}\|u_{j}\|_{X_{0,b}}.\label{3.010}
\end{eqnarray}
{\bf Case 3}: $N_{1}\geq 10$, $N_{1}\sim N_{2}\gg N$, we consider the following cases: {\bf Case 3A}: $N\leq 8$ and {\bf Case
3B}: $N\geq 8$.

\noindent {\bf Case 3A}: $N_{1}\geq 10$, $N_{1}\sim N_{2}\gg N$, $N\leq 8$.
By using H\"{o}lder inequality and \eqref{2.01}, $s\geq -\frac{1}{8}+\frac{\epsilon}{2}$, we have
\begin{eqnarray}
&&|I_{1}|\leq C\sum N\langle N\rangle^{s+a-2}N_{1}^{-s}N_{2}^{-s}
\|P_{N_{1}}u_{1}P_{N_{2}}u_{2}\|_{L_{xT}^{2}}\|P_{N}h\|_{L_{xT}^{2}}\nonumber\\
&&\leq C\sum NN_{1}^{-\epsilon}N_{1}^{-\frac{1}{4}-2s+\epsilon}
\|P_{N_{1}}u_{1}\|_{X_{0,b}}\|P_{N_{2}}u_{2}\|_{X_{0,b}}\|P_{N}h\|_{L_{xT}^{2}}\nonumber\\
&&\leq C\sum NN_{1}^{-\epsilon}\|P_{N_{1}}u_{1}\|_{X_{0,b}}\|P_{N_{2}}u_{2}\|_{X_{0,b}}\|P_{N}h\|_{L_{xT}^{2}}
\leq C\|h\|_{L_{xT}^{2}}\prod_{j=1}^{2}\|u_{j}\|_{X_{0,b}}.\label{3.011}
\end{eqnarray}
{\bf Case 3B}: $N_{1}\geq 10$, $N_{1}\sim N_{2}\gg N$, $N\geq 8$, we consider the following cases: {\bf Case 3$B_{1}$}: $0\leq a\leq 1-s$, $-\frac{1}{8}+\frac{\epsilon}{2}\leq s\leq 1$ and {\bf Case
3$B_{2}$}: $a\geq 1-s$.

\noindent {\bf Case 3$B_{1}$}: When $0\leq a\leq 1-s$, $-\frac{1}{8}+\frac{\epsilon}{2}\leq s\leq 1$,
by using H\"{o}lder inequality and \eqref{2.01}, we have
\begin{eqnarray}
&&|I_{1}|\leq C\sum NN^{s+a-2}N_{1}^{-s}N_{2}^{-s}
\|P_{N_{1}}u_{1}P_{N_{2}}u_{2}\|_{L_{xT}^{2}}\|P_{N}h\|_{L_{xT}^{2}}\nonumber\\
&&\leq C\sum N_{1}^{-\epsilon}N^{s+a-1}N_{1}^{-2s-\frac{1}{4}+\epsilon}
\|P_{N_{1}}u_{1}\|_{X_{0,b}}\|P_{N_{2}}u_{2}\|_{X_{0,b}}\|P_{N}h\|_{L_{xT}^{2}}\nonumber\\
&&\leq C\sum N_{1}^{-\epsilon}\|P_{N_{1}}u_{1}\|_{X_{0,b}}\|P_{N_{2}}u_{2}\|_{X_{0,b}}\|P_{N}h\|_{L_{xT}^{2}}\nonumber\\
&&\leq C\|h\|_{L_{xT}^{2}}\prod_{j=1}^{2}\|u_{j}\|_{X_{0,b}}.\label{3.012}
\end{eqnarray}
{\bf Case 3$B_{2}$}: When $s\geq-\frac{1}{8}+\frac{\epsilon}{2}$, $1-s\leq a\leq s+\frac{5}{4}-\epsilon$,
by using H\"{o}lder inequality and \eqref{2.01}, we have
\begin{eqnarray}
&&|I_{1}|\leq C\sum NN^{s+a-2}N_{1}^{-s}N_{2}^{-s}
\|P_{N_{1}}u_{1}P_{N_{2}}u_{2}\|_{L_{xT}^{2}}\|P_{N}h\|_{L_{xT}^{2}}\nonumber\\
&&\leq C\sum N_{1}^{-\epsilon}N^{s+a-1}N_{1}^{-2s-\frac{1}{4}+\epsilon}
\|P_{N_{1}}u_{1}\|_{X_{0,b}}\|P_{N_{2}}u_{2}\|_{X_{0,b}}\|P_{N}h\|_{L_{xT}^{2}}\nonumber\\
&&\leq C\sum N_{1}^{-\epsilon}N_{1}^{a-s-1-\frac{1}{4}+\epsilon}\|P_{N_{1}}u_{1}\|_{X_{0,b}}\|P_{N_{2}}u_{2}\|_{X_{0,b}}\|P_{N}h\|_{L_{xT}^{2}}\nonumber\\
&&\leq C\sum N_{1}^{-\epsilon}\|P_{N_{1}}u_{1}\|_{X_{0,b}}\|P_{N_{2}}u_{2}\|_{X_{0,b}}\|P_{N}h\|_{L_{xT}^{2}}\nonumber\\
&&\leq C\|h\|_{L_{xT}^{2}}\prod_{j=1}^{2}\|u_{j}\|_{X_{0,b}}.\label{3.013}
\end{eqnarray}
{\bf Case 4}: $N_{1}\geq 10$, $N_{1}\sim N_{2}\sim N$, by using H\"{o}lder inequality and \eqref{2.01}, $s\geq-\frac{1}{8}+\frac{\epsilon}{2}$, $0\leq a\leq s+\frac{5}{4}-\epsilon$, we have
\begin{eqnarray}
&&|I_{1}|\leq C\sum NN^{s+a-2}N_{1}^{-s}N_{2}^{-s}
\|P_{N_{1}}u_{1}P_{N_{2}}u_{2}\|_{L_{xT}^{2}}\|P_{N}h\|_{L_{xT}^{2}}\nonumber\\
&&\leq C\sum N_{1}^{-\epsilon}N_{1}^{a-1-s-\frac{1}{4}+\epsilon}
\|P_{N_{1}}u_{1}\|_{X_{0,b}}\|P_{N_{2}}u_{2}\|_{X_{0,b}}\|P_{N}h\|_{L_{xT}^{2}}\nonumber\\
&&\leq C\sum N_{1}^{-\epsilon}\|P_{N_{1}}u_{1}\|_{X_{0,b}}\|P_{N_{2}}u_{2}\|_{X_{0,b}}\|P_{N}h\|_{L_{xT}^{2}}\nonumber\\
&&\leq C\|h\|_{L_{xT}^{2}}\prod_{j=1}^{2}\|u_{j}\|_{X_{0,b}}.\label{3.014}
\end{eqnarray}

This completes the proof of Lemma 3.1.

\begin{Lemma}\label{Lemma3.2}
Let $2\leq n\leq 6$, $0<\epsilon\ll1$, $0<T<1$, $s\geq\frac{n-3}{4}+2\epsilon$, $b_{0}=-\frac{1}{2}+\frac{\epsilon}{12}$, $b=\frac{1}{2}+\frac{\epsilon}{24}$ and
\begin{align*}
0\leq a &\leq \min\left\{1-2\epsilon, s+\frac{5-n}{2}-4\epsilon\right\}\\
&=
\begin{cases}
1-2\epsilon,&n=2,\\
\min\left\{1-2\epsilon, s+\frac{5-n}{2}-4\epsilon\right\},&3\leq n\leq 6.
\end{cases}
\end{align*}
Then, we have
\begin{eqnarray}
&&\left\|\eta\left(\frac{t}{T}\right)\mathscr{F}_{xt}^{-1}\left(\frac{|\xi|^{2}\mathscr{F}_{xt}(u^{2})}{2i\phi_{\pm}(\xi)}\right)\right\|_{X_{s+a,b_{0}}}\leq CT^{-b_{0}}\|u\|_{X_{s,b}}^{2}.\label{3.015}
\end{eqnarray}
\end{Lemma}
\noindent {\bf Proof.}
Inspired by \cite{ET2025}, we prove Lemma 3.2.
From \eqref{3.02}, to prove \eqref{3.015}, it suffices to show that
\begin{eqnarray}
&&\left\|\eta\left(\frac{t}{T}\right)|D|\langle D\rangle^{-2}(u^{2})\right\|_{X_{s+a,b_{0}}}\leq CT^{-b_{0}}\|u\|_{X_{s,b}}^{2}.\label{3.016}
\end{eqnarray}
By using \cite[Lemma 2.11]{T2006}, we have
\begin{eqnarray}
&&\left\|\eta\left(\frac{t}{T}\right)|D|\langle D\rangle^{-2}(u^{2})\right\|_{X_{s+a,b_{0}}}\leq CT^{-b_{0}}
\left\|\eta\left(\frac{t}{T}\right)|D|\langle D\rangle^{-2}(u^{2})\right\|_{X_{s+a,0}}\nonumber\\
&&=CT^{-b_{0}}\left\||D|\langle D\rangle^{-2}(u^{2})\right\|_{L_{T}^{2}H_{x}^{s+a}}.\label{3.017}
\end{eqnarray}
By using \eqref{3.017}, to prove \eqref{3.015}, it suffices to show that
\begin{eqnarray}
&&\left\||D|\langle D\rangle^{s+a-2}(u^{2})\right\|_{L_{xT}^{2}}\leq C\|u\|_{X_{s,b}}^{2}.\label{3.018}
\end{eqnarray}
According to the duality idea, to prove \eqref{3.018}, we only need to prove
\begin{eqnarray}
&&\left|\int_{0}^{T}\int_{\SR^{n}}|D|\langle D\rangle^{s+a-2}\left(\langle D\rangle^{-s} u_{1}\langle D\rangle^{-s} u_{2}\right)(x,t)
\overline{h}(x,t)dxdt\right|\nonumber\\
&&\leq C\|h\|_{L_{xT}^{2}}\prod_{j=1}^{2}\|u_{j}\|_{X_{0,b}}.\label{3.019}
\end{eqnarray}
We denote
\begin{eqnarray}
&&I_{2}:=
\int_{0}^{T}\int_{\SR^{n}}|D|\langle D\rangle^{s+a-2}\left(\langle D\rangle^{-s} u_{1}\langle D\rangle^{-s} u_{2}\right)(x,t)\overline{h}(x,t)dxdt.\label{3.020}
\end{eqnarray}
Let $P_{N}h$ and $P_{N_{j}}u_{j}(j=1,2)$ be the dyadic decompositions of $h$ and $u_{j}$, respectively.
Then, we have that $\supp \mathscr{F}_{x}P_{N}h\subset\{\xi\in\R^{n}:|\xi|\sim N\}$, and $\supp \mathscr{F}_{x}P_{N_{j}}u_{j}\subset \{\xi\in\R^{n}:|\xi|\sim N_{j}\}$, where $N$ and $N_{j}$ are dyadic numbers. We define $\sum=\sum\limits_{N_{1}, N_{2}, N}$. Without loss of generality, we assume $N_{1}\geq N_{2}$.

\noindent We consider the following cases: {\bf Case 1}: $N_{1}\leq 10$; {\bf Case
2}: $N_{1}\geq 10$, $N_{1}\gg N_{2}$ and {\bf Case 3}: $N_{1}\geq 10$, $N_{1}\sim N_{2}\gg N$; {\bf Case 4}: $N_{1}\geq 10$, $N_{1}\sim N_{2}\sim N$.

\noindent {\bf Case 1}: When $N_{1}\leq 10$, by using H\"{o}lder inequality and Bernstein inequality, we have
\begin{eqnarray}
&&|I_{2}|\leq C\sum N\langle N\rangle^{s+a-2}\langle N_{1}\rangle^{-s}\langle N_{2}\rangle^{-s}
\|P_{N_{1}}u_{1}\|_{L_{xT}^{2}}\|P_{N_{2}}u_{2}\|_{L_{xT}^{\infty}}\|P_{N}h\|_{L_{xT}^{2}}\nonumber\\
&&\leq C\sum  NN_{2}^{\frac{n}{2}}\langle N_{1}\rangle^{-s}\langle N_{2}\rangle^{-s}
\|P_{N_{1}}u_{1}\|_{X_{0,b}}\|P_{N_{2}}u_{2}\|_{X_{0,b}}\|P_{N}h\|_{L_{xT}^{2}}\nonumber\\
&&\leq C\sum NN_{1}^{\frac{n}{4}}N_{2}^{\frac{n}{4}}\|P_{N_{1}}u_{1}\|_{X_{0,b}}\|P_{N_{2}}u_{2}\|_{X_{0,b}}\|P_{N}h\|_{L_{xT}^{2}}\nonumber\\
&&\leq C\|h\|_{L_{xT}^{2}}\prod_{j=1}^{2}\|u_{j}\|_{X_{0,b}}.\label{3.021}
\end{eqnarray}
{\bf Case 2}: When $N_{1}\geq 10$,  $N_{1}\gg N_{2}$, we consider the following cases: {\bf Case 2A}: $N_{2}\leq 8$ and {\bf Case
2B}: $N_{2}\geq 8$.

\noindent {\bf Case 2A}: $N_{1}\geq 10$,  $N_{1}\gg N_{2}$, $N_{2}\leq 8$, we have $N_{1}\sim N$.
By using H\"{o}lder inequality and Bernstein inequality, $0\leq a\leq 1-2\epsilon$, we have
\begin{eqnarray}
&&|I_{2}|\leq C\sum NN^{s+a-2}N_{1}^{-s}\langle N_{2}\rangle^{-s}
\|P_{N_{1}}u_{1}\|_{L_{xT}^{2}}\|P_{N_{2}}u_{2}\|_{L_{xT}^{\infty}}\|P_{N}h\|_{L_{xT}^{2}}\nonumber\\
&&\leq C\sum N^{a-1}N_{2}^{\frac{n}{2}}\langle N_{2}\rangle^{-s}
\|P_{N_{1}}u_{1}\|_{X_{0,b}}\|P_{N_{2}}u_{2}\|_{X_{0,b}}\|P_{N}h\|_{L_{xT}^{2}}\nonumber\\
&&\leq C\sum N^{a-1}N_{2}^{\frac{n}{2}}\|P_{N_{1}}u_{1}\|_{X_{0,b}}\|P_{N_{2}}u_{2}\|_{X_{0,b}}\|P_{N}h\|_{L_{xT}^{2}}\nonumber\\
&&\leq C\|h\|_{L_{xT}^{2}}\prod_{j=1}^{2}\|u_{j}\|_{X_{0,b}}.\label{3.022}
\end{eqnarray}
{\bf Case 2B}: $N_{1}\geq 10$,  $N_{1}\gg N_{2}$, $N_{2}\geq 8$, we have $N_{1}\sim N$.
By using H\"{o}lder inequality and \eqref{2.02}, $2\leq n\leq 6$, $s\geq\frac{n-3}{4}+2\epsilon$, $0\leq a\leq \min\{1-2\epsilon,s+\frac{5-n}{2}-4\epsilon\}$, we have
\begin{eqnarray}
&&|I_{2}|\leq C\sum NN^{s+a-2}N_{1}^{-s}N_{2}^{-s}
\|P_{N_{1}}u_{1}P_{N_{2}}u_{2}\|_{L_{xT}^{2}}\|P_{N}h\|_{L_{xT}^{2}}\nonumber\\
&&\leq C\sum N_{1}^{-\epsilon}N_{1}^{a-1+\epsilon}N_{1}^{-\frac{1}{2}+\epsilon}N_{2}^{\frac{n}{2}-1+2\epsilon}N_{2}^{-s}
\|P_{N_{1}}u_{1}\|_{X_{0,b}}\|P_{N_{2}}u_{2}\|_{X_{0,b}}\|P_{N}h\|_{L_{xT}^{2}}\nonumber\\
&&\leq C\sum N_{1}^{-\epsilon} N_{2}^{-s+a-\frac{5}{2}+\frac{n}{2}+4\epsilon}
\|P_{N_{1}}u_{1}\|_{X_{0,b}}\|P_{N_{2}}u_{2}\|_{X_{0,b}}\|P_{N}h\|_{L_{xT}^{2}}\nonumber\\
&&\leq C\|h\|_{L_{xT}^{2}}\prod_{j=1}^{2}\|u_{j}\|_{X_{0,b}}.\label{3.023}
\end{eqnarray}
{\bf Case 3}: $N_{1}\geq 10$, $N_{1}\sim N_{2}\gg N$, we consider the following cases: {\bf Case 3A}: $N\leq 8$ and {\bf Case
3B}: $N\geq 8$.

\noindent {\bf Case 3A}: $N_{1}\geq 10$, $N_{1}\sim N_{2}\gg N$, $N\leq 8$.
By using H\"{o}lder inequality and \eqref{2.02}, $s\geq \frac{n-3}{4}+2\epsilon$, we have
\begin{eqnarray}
&&|I_{2}|\leq C\sum N\langle N\rangle^{s+a-2}N_{1}^{-s}N_{2}^{-s}
\|P_{N_{1}}u_{1}P_{N_{2}}u_{2}\|_{L_{xT}^{2}}\|P_{N}h\|_{L_{xT}^{2}}\nonumber\\
&&\leq C\sum NN_{1}^{-\epsilon}N_{1}^{-\frac{3}{2}+\frac{n}{2}-2s+4\epsilon}
\|P_{N_{1}}u_{1}\|_{X_{0,b}}\|P_{N_{2}}u_{2}\|_{X_{0,b}}\|P_{N}h\|_{L_{xT}^{2}}\nonumber\\
&&\leq C\sum NN_{1}^{-\epsilon}\|P_{N_{1}}u_{1}\|_{X_{0,b}}\|P_{N_{2}}u_{2}\|_{X_{0,b}}\|P_{N}h\|_{L_{xT}^{2}}\nonumber\\
&&\leq C\|h\|_{L_{xT}^{2}}\prod_{j=1}^{2}\|u_{j}\|_{X_{0,b}}.\label{3.024}
\end{eqnarray}
{\bf Case 3B}: $N_{1}\geq 10$, $N_{1}\sim N_{2}\gg N$, $N\geq 8$, we consider the following cases: {\bf Case 3$B_{1}$}: $0\leq a\leq 1-s$, $\frac{n-3}{4}+2\epsilon\leq s\leq 1$ and {\bf Case
3$B_{2}$}: $a\geq 1-s$.

\noindent {\bf Case 3$B_{1}$}: When $0\leq a\leq 1-s$, $\frac{n-3}{4}+2\epsilon\leq s\leq 1$,
by using H\"{o}lder inequality and \eqref{2.02}, we have
\begin{eqnarray}
&&|I_{2}|\leq C\sum NN^{s+a-2}N_{1}^{-s}N_{2}^{-s}
\|P_{N_{1}}u_{1}P_{N_{2}}u_{2}\|_{L_{xT}^{2}}\|P_{N}h\|_{L_{xT}^{2}}\nonumber\\
&&\leq C\sum N_{1}^{-\epsilon}N^{s+a-1}N_{1}^{-2s-\frac{3}{2}+\frac{n}{2}+4\epsilon}
\|P_{N_{1}}u_{1}\|_{X_{0,b}}\|P_{N_{2}}u_{2}\|_{X_{0,b}}\|P_{N}h\|_{L_{xT}^{2}}\nonumber\\
&&\leq C\sum N_{1}^{-\epsilon}\|P_{N_{1}}u_{1}\|_{X_{0,b}}\|P_{N_{2}}u_{2}\|_{X_{0,b}}\|P_{N}h\|_{L_{xT}^{2}}\nonumber\\
&&\leq C\|h\|_{L_{xT}^{2}}\prod_{j=1}^{2}\|u_{j}\|_{X_{0,b}}.\label{3.025}
\end{eqnarray}
{\bf Case 3$B_{2}$}: When $s\geq\frac{n-3}{4}+2\epsilon$,  $1-s\le a\leq s+\frac{5-n}{2}-4\epsilon$,
by using H\"{o}lder inequality and \eqref{2.02}, we have
\begin{eqnarray}
&&|I_{2}|\leq C\sum NN^{s+a-2}N_{1}^{-s}N_{2}^{-s}
\|P_{N_{1}}u_{1}P_{N_{2}}u_{2}\|_{L_{xT}^{2}}\|P_{N}h\|_{L_{xT}^{2}}\nonumber\\
&&\leq C\sum N_{1}^{-\epsilon}N^{s+a-1}N_{1}^{-2s-\frac{3-n}{2}+4\epsilon}
\|P_{N_{1}}u_{1}\|_{X_{0,b}}\|P_{N_{2}}u_{2}\|_{X_{0,b}}\|P_{N}h\|_{L_{xT}^{2}}\nonumber\\
&&\leq C\sum N_{1}^{-\epsilon}N_{1}^{a-s-1-\frac{3-n}{2}+4\epsilon}\|P_{N_{1}}u_{1}\|_{X_{0,b}}\|P_{N_{2}}u_{2}\|_{X_{0,b}}\|P_{N}h\|_{L_{xT}^{2}}\nonumber\\
&&\leq C\sum N_{1}^{-\epsilon}\|P_{N_{1}}u_{1}\|_{X_{0,b}}\|P_{N_{2}}u_{2}\|_{X_{0,b}}\|P_{N}h\|_{L_{xT}^{2}}\nonumber\\
&&\leq C\|h\|_{L_{xT}^{2}}\prod_{j=1}^{2}\|u_{j}\|_{X_{0,b}}.\label{3.026}
\end{eqnarray}
{\bf Case 4}: $N_{1}\geq 10$, $N_{1}\sim N_{2}\sim N$, by using H\"{o}lder inequality and \eqref{2.02},  $s\geq\frac{n-3}{4}+2\epsilon$, $0\leq a\leq s+\frac{5-n}{2}-4\epsilon$, we have
\begin{eqnarray}
&&|I_{2}|\leq C\sum NN^{s+a-2}N_{1}^{-s}N_{2}^{-s}
\|P_{N_{1}}u_{1}P_{N_{2}}u_{2}\|_{L_{xT}^{2}}\|P_{N}h\|_{L_{xT}^{2}}\nonumber\\
&&\leq C\sum N_{1}^{-\epsilon}N_{1}^{a-1-s-\frac{3-n}{2}+4\epsilon}
\|P_{N_{1}}u_{1}\|_{X_{0,b}}\|P_{N_{2}}u_{2}\|_{X_{0,b}}\|P_{N}h\|_{L_{xT}^{2}}\nonumber\\
&&\leq C\sum N_{1}^{-\epsilon}\|P_{N_{1}}u_{1}\|_{X_{0,b}}\|P_{N_{2}}u_{2}\|_{X_{0,b}}\|P_{N}h\|_{L_{xT}^{2}}\nonumber\\
&&\leq C\|h\|_{L_{xT}^{2}}\prod_{j=1}^{2}\|u_{j}\|_{X_{0,b}}.\label{3.027}
\end{eqnarray}

This completes the proof of Lemma 3.2.

\begin{Remark}
The restriction on $n$ arises from the fact that in \eqref{3.025}, we require $\frac{n-3}{4}+2\epsilon\leq s\leq 1$, which forces the restriction $n\leq 6$.
\end{Remark}

\begin{Lemma}\label{Lemma3.3}
Let $n\geq 7$, $0<T<1$, $0<\epsilon\ll1$, $s\geq\frac{n-5}{2}+4\epsilon$, $b_{0}=-\frac{1}{2}+\frac{\epsilon}{12}$, $b=\frac{1}{2}+\frac{\epsilon}{24}$ and
\begin{eqnarray*}
&&0\leq a\leq \min\left\{1-2\epsilon, s+\frac{5-n}{2}-4\epsilon\right\}.
\end{eqnarray*}
Then, we have
\begin{eqnarray}
&&\left\|\eta\left(\frac{t}{T}\right)\mathscr{F}_{xt}^{-1}\left(\frac{|\xi|^{2}\mathscr{F}_{xt}(u^{2})}{2i\phi_{\pm}(\xi)}\right)\right\|_{X_{s+a,b_{0}}}\leq CT^{-b_{0}}\|u\|_{X_{s,b}}^{2}.\label{3.028}
\end{eqnarray}
\end{Lemma}
\noindent {\bf Proof.}
Inspired by \cite{ET2025}, we prove Lemma 3.3.
From \eqref{3.02}, to prove \eqref{3.028}, it suffices to show that
\begin{eqnarray}
&&\left\|\eta\left(\frac{t}{T}\right)|D|\langle D\rangle^{-2}(u^{2})\right\|_{X_{s+a,b_{0}}}\leq CT^{-b_{0}}\|u\|_{X_{s,b}}^{2}.\label{3.029}
\end{eqnarray}
By using \cite[Lemma 2.11]{T2006}, we have
\begin{eqnarray}
&&\left\|\eta\left(\frac{t}{T}\right)|D|\langle D\rangle^{-2}(u^{2})\right\|_{X_{s+a,b_{0}}}\leq CT^{-b_{0}}
\left\|\eta\left(\frac{t}{T}\right)|D|\langle D\rangle^{-2}(u^{2})\right\|_{X_{s+a,0}}\nonumber\\
&&=CT^{-b_{0}}\left\||D|\langle D\rangle^{-2}(u^{2})\right\|_{L_{T}^{2}H_{x}^{s+a}}.\label{3.030}
\end{eqnarray}
By using \eqref{3.030}, to prove \eqref{3.028}, it suffices to show that
\begin{eqnarray}
&&\left\||D|\langle D\rangle^{s+a-2}(u^{2})\right\|_{L_{xT}^{2}}\leq C\|u\|_{X_{s,b}}^{2}.\label{3.031}
\end{eqnarray}
According to the duality idea, to prove \eqref{3.031}, we only need to prove
\begin{eqnarray}
&&\left|\int_{0}^{T}\int_{\SR^{n}}|D|\langle D\rangle^{s+a-2}\left(\langle D\rangle^{-s} u_{1}\langle D\rangle^{-s} u_{2}\right)(x,t)
\overline{h}(x,t)dxdt\right|\nonumber\\
&&\leq C\|h\|_{L_{xT}^{2}}\prod_{j=1}^{2}\|u_{j}\|_{X_{0,b}}.\label{3.032}
\end{eqnarray}
We denote
\begin{eqnarray}
&&I_{3}:=
\int_{0}^{T}\int_{\SR^{n}}|D|\langle D\rangle^{s+a-2}\left(\langle D\rangle^{-s} u_{1}\langle D\rangle^{-s} u_{2}\right)(x,t)\overline{h}(x,t)dxdt.\label{3.033}
\end{eqnarray}
Let $P_{N}h$ and $P_{N_{j}}u_{j}(j=1,2)$ be the dyadic decompositions of $h$ and $u_{j}$, respectively.
Then, we have that $\supp \mathscr{F}_{x}P_{N}h\subset\{\xi\in\R^{n}:|\xi|\sim N\}$, and $\supp \mathscr{F}_{x}P_{N_{j}}u_{j}\subset \{\xi\in\R^{n}:|\xi|\sim N_{j}\}$, where $N$ and $N_{j}$ are dyadic numbers. We define $\sum=\sum\limits_{N_{1}, N_{2}, N}$. Without loss of generality, we assume $N_{1}\geq N_{2}$.

\noindent We consider the following cases: {\bf Case 1}: $N_{1}\leq 10$; {\bf Case
2}: $N_{1}\geq 10$, $N_{1}\gg N_{2}$ and {\bf Case 3}: $N_{1}\geq 10$, $N_{1}\sim N_{2}\gg N$; {\bf Case 4}: $N_{1}\geq 10$, $N_{1}\sim N_{2}\sim N$.

\noindent {\bf Case 1}: When $N_{1}\leq 10$, by using H\"{o}lder inequality and Bernstein inequality, we have
\begin{eqnarray}
&&|I_{3}|\leq C\sum N\langle N\rangle^{s+a-2}\langle N_{1}\rangle^{-s}\langle N_{2}\rangle^{-s}
\|P_{N_{1}}u_{1}\|_{L_{xT}^{2}}\|P_{N_{2}}u_{2}\|_{L_{xT}^{\infty}}\|P_{N}h\|_{L_{xT}^{2}}\nonumber\\
&&\leq C\sum NN_{2}^{\frac{n}{2}}\langle N_{1}\rangle^{-s}\langle N_{2}\rangle^{-s}
\|P_{N_{1}}u_{1}\|_{X_{0,b}}\|P_{N_{2}}u_{2}\|_{X_{0,b}}\|P_{N}h\|_{L_{xT}^{2}}\nonumber\\
&&\leq C\sum NN_{1}^{\frac{n}{4}}N_{2}^{\frac{n}{4}}\|P_{N_{1}}u_{1}\|_{X_{0,b}}\|P_{N_{2}}u_{2}\|_{X_{0,b}}\|P_{N}h\|_{L_{xT}^{2}}\nonumber\\
&&\leq C\|h\|_{L_{xT}^{2}}\prod_{j=1}^{2}\|u_{j}\|_{X_{0,b}}.\label{3.034}
\end{eqnarray}
{\bf Case 2}: When $N_{1}\geq 10$,  $N_{1}\gg N_{2}$, we consider the following cases: {\bf Case 2A}: $N_{2}\leq 8$ and {\bf Case
2B}: $N_{2}\geq 8$.

\noindent {\bf Case 2A}: $N_{1}\geq 10$,  $N_{1}\gg N_{2}$, $N_{2}\leq 8$, we have $N_{1}\sim N$.
By using H\"{o}lder inequality and Bernstein inequality, $0\leq a\leq 1-2\epsilon$, we have
\begin{eqnarray}
&&|I_{3}|\leq C\sum NN^{s+a-2}N_{1}^{-s}\langle N_{2}\rangle^{-s}
\|P_{N_{1}}u_{1}\|_{L_{xT}^{2}}\|P_{N_{2}}u_{2}\|_{L_{xT}^{\infty}}\|P_{N}h\|_{L_{xT}^{2}}\nonumber\\
&&\leq C\sum N^{a-1}N_{2}^{\frac{n}{2}}\langle N_{2}\rangle^{-s}
\|P_{N_{1}}u_{1}\|_{X_{0,b}}\|P_{N_{2}}u_{2}\|_{X_{0,b}}\|P_{N}h\|_{L_{xT}^{2}}\nonumber\\
&&\leq C\sum N^{a-1}N_{2}^{\frac{n}{2}}\|P_{N_{1}}u_{1}\|_{X_{0,b}}\|P_{N_{2}}u_{2}\|_{X_{0,b}}\|P_{N}h\|_{L_{xT}^{2}}\nonumber\\
&&\leq C\|h\|_{L_{xT}^{2}}\prod_{j=1}^{2}\|u_{j}\|_{X_{0,b}}.\label{3.035}
\end{eqnarray}
{\bf Case 2B}: When $N_{1}\geq 10$,  $N_{1}\gg N_{2}$, $N_{2}\geq 8$, we have $N_{1}\sim N$.
By using H\"{o}lder inequality and \eqref{2.02}, $s\geq\frac{n-5}{2}+4\epsilon$, $0\leq a\leq \min\{1-2\epsilon,s+\frac{5-n}{2}-4\epsilon\}$, we have
\begin{eqnarray}
&&|I_{3}|\leq C\sum NN^{s+a-2}N_{1}^{-s}N_{2}^{-s}
\|P_{N_{1}}u_{1}P_{N_{2}}u_{2}\|_{L_{xT}^{2}}\|P_{N}h\|_{L_{xT}^{2}}\nonumber\\
&&\leq C\sum N_{1}^{-\epsilon}N_{1}^{a-1+\epsilon}N_{1}^{-\frac{1}{2}+\epsilon}N_{2}^{\frac{n}{2}-1+2\epsilon}N_{2}^{-s}
\|P_{N_{1}}u_{1}\|_{X_{0,b}}\|P_{N_{2}}u_{2}\|_{X_{0,b}}\|P_{N}h\|_{L_{xT}^{2}}\nonumber\\
&&\leq C\sum N_{1}^{-\epsilon} N_{2}^{-s+a-\frac{5}{2}+\frac{n}{2}+2\epsilon}
\|P_{N_{1}}u_{1}\|_{X_{0,b}}\|P_{N_{2}}u_{2}\|_{X_{0,b}}\|P_{N}h\|_{L_{xT}^{2}}\nonumber\\
&&\leq C\|h\|_{L_{xT}^{2}}\prod_{j=1}^{2}\|u_{j}\|_{X_{0,b}}.\label{3.036}
\end{eqnarray}
{\bf Case 3}: $N_{1}\geq 10$, $N_{1}\sim N_{2}\gg N$, we consider the following  cases: {\bf Case 3A}: $N\leq 8$ and {\bf Case
3B}: $N\geq 8$.

\noindent {\bf Case 3A}: $N_{1}\geq 10$, $N_{1}\sim N_{2}\gg N$, $N\leq 8$.
By using H\"{o}lder inequality and \eqref{2.02}, $n\geq 7$, $s\geq\frac{n-5}{2}+4\epsilon$, we have
\begin{eqnarray}
&&|I_{3}|\leq C\sum N\langle N\rangle^{s+a-2}N_{1}^{-s}N_{2}^{-s}
\|P_{N_{1}}u_{1}P_{N_{2}}u_{2}\|_{L_{xT}^{2}}\|P_{N}h\|_{L_{xT}^{2}}\nonumber\\
&&\leq C\sum NN_{1}^{-\epsilon}N_{1}^{-\frac{3}{2}+\frac{n}{2}-2s+4\epsilon}
\|P_{N_{1}}u_{1}\|_{X_{0,b}}\|P_{N_{2}}u_{2}\|_{X_{0,b}}\|P_{N}h\|_{L_{xT}^{2}}\nonumber\\
&&\leq C\sum NN_{1}^{-\epsilon}\|P_{N_{1}}u_{1}\|_{X_{0,b}}\|P_{N_{2}}u_{2}\|_{X_{0,b}}\|P_{N}h\|_{L_{xT}^{2}}\nonumber\\
&&\leq C\|h\|_{L_{xT}^{2}}\prod_{j=1}^{2}\|u_{j}\|_{X_{0,b}}.\label{3.037}
\end{eqnarray}
{\bf Case 3B}: $N_{1}\geq 10$, $N_{1}\sim N_{2}\gg N$, $N\geq 8$, noting that $s\geq \frac{n-5}{2}+4\epsilon>1$, $n\geq 7$,
by using H\"{o}lder inequality and \eqref{2.02}, $0\leq a\leq s+\frac{5-n}{2}-4\epsilon$, we have
\begin{eqnarray}
&&|I_{3}|\leq C\sum NN^{s+a-2}N_{1}^{-s}N_{2}^{-s}
\|P_{N_{1}}u_{1}P_{N_{2}}u_{2}\|_{L_{xT}^{2}}\|P_{N}h\|_{L_{xT}^{2}}\nonumber\\
&&\leq C\sum N_{1}^{-\epsilon}N^{s+a-1}N_{1}^{-2s-\frac{3-n}{2}+4\epsilon}
\|P_{N_{1}}u_{1}\|_{X_{0,b}}\|P_{N_{2}}u_{2}\|_{X_{0,b}}\|P_{N}h\|_{L_{xT}^{2}}\nonumber\\
&&\leq C\sum N_{1}^{-\epsilon}N_{1}^{a-s-1-\frac{3-n}{2}+4\epsilon}\|P_{N_{1}}u_{1}\|_{X_{0,b}}\|P_{N_{2}}u_{2}\|_{X_{0,b}}\|P_{N}h\|_{L_{xT}^{2}}\nonumber\\
&&\leq C\sum N_{1}^{-\epsilon}\|P_{N_{1}}u_{1}\|_{X_{0,b}}\|P_{N_{2}}u_{2}\|_{X_{0,b}}\|P_{N}h\|_{L_{xT}^{2}}\nonumber\\
&&\leq C\|h\|_{L_{xT}^{2}}\prod_{j=1}^{2}\|u_{j}\|_{X_{0,b}}.\label{3.038}
\end{eqnarray}
{\bf Case 4}: $N_{1}\geq 10$, $N_{1}\sim N_{2}\sim N$, by using H\"{o}lder inequality and \eqref{2.02}, $s\geq\frac{n-5}{2}+4\epsilon$, $0\leq a\leq \min\{1-2\epsilon,s+\frac{5-n}{2}-4\epsilon\}$, we have
\begin{eqnarray}
&&|I_{3}|\leq C\sum NN^{s+a-2}N_{1}^{-s}N_{2}^{-s}
\|P_{N_{1}}u_{1}P_{N_{2}}u_{2}\|_{L_{xT}^{2}}\|P_{N}h\|_{L_{xT}^{2}}\nonumber\\
&&\leq C\sum N_{1}^{-\epsilon}N_{1}^{a-1-s-\frac{3-n}{2}+4\epsilon}
\|P_{N_{1}}u_{1}\|_{X_{0,b}}\|P_{N_{2}}u_{2}\|_{X_{0,b}}\|P_{N}h\|_{L_{xT}^{2}}\nonumber\\
&&\leq C\sum N_{1}^{-\epsilon}\|P_{N_{1}}u_{1}\|_{X_{0,b}}\|P_{N_{2}}u_{2}\|_{X_{0,b}}\|P_{N}h\|_{L_{xT}^{2}}\nonumber\\
&&\leq C\|h\|_{L_{xT}^{2}}\prod_{j=1}^{2}\|u_{j}\|_{X_{0,b}}.\label{3.039}
\end{eqnarray}

This completes the proof of Lemma 3.3.

\begin{Lemma}\label{Lemma3.4}
Let $n=1$, $0<T<1$, $0<\epsilon\ll1$, $s\geq-\frac{1}{6}+\frac{\epsilon}{2}$, $b_{0}=-\frac{1}{2}+\frac{\epsilon}{12}$, $b=\frac{1}{2}+\frac{\epsilon}{24}$ and
\begin{eqnarray*}
&&0\leq a\leq \min\{1-2\epsilon,s+\frac{5}{4}-\epsilon,2s+\frac{3}{2}-\epsilon\}=1-2\epsilon.
\end{eqnarray*}
Then, we have
\begin{eqnarray}
&&\left\|\eta\left(\frac{t}{T}\right)\mathscr{F}_{xt}^{-1}\left(\frac{|\xi|^{2}\mathscr{F}_{xt}(u^{3})}{2i\phi_{\pm}(\xi)}\right)\right\|_{X_{s+a,b_{0}}}\leq CT^{-b_{0}}\|u\|_{X_{s,b}}^{3}.\label{3.040}
\end{eqnarray}
\end{Lemma}
\noindent {\bf Proof.}
Inspired by \cite{ET2025}, we prove Lemma 3.4.
From \eqref{3.02}, to prove \eqref{3.040}, it suffices to show that
\begin{eqnarray}
&&\left\|\eta\left(\frac{t}{T}\right)|D|\langle D\rangle^{-2}(u^{3})\right\|_{X_{s+a,b_{0}}}\leq CT^{-b_{0}}\|u\|_{X_{s,b}}^{3}.\label{3.041}
\end{eqnarray}
By using \cite[Lemma 2.11]{T2006}, we have
\begin{eqnarray}
&&\left\|\eta\left(\frac{t}{T}\right)|D|\langle D\rangle^{-2}(u^{3})\right\|_{X_{s+a,b_{0}}}\leq CT^{-b_{0}}
\left\|\eta\left(\frac{t}{T}\right)|D|\langle D\rangle^{-2}(u^{3})\right\|_{X_{s+a,0}}\nonumber\\
&&=CT^{-b_{0}}\left\||D|\langle D\rangle^{-2}(u^{3})\right\|_{L_{T}^{2}H_{x}^{s+a}}.\label{3.042}
\end{eqnarray}
By using \eqref{3.042}, to prove \eqref{3.040}, it suffices to show that
\begin{eqnarray}
&&\left\||D|\langle D\rangle^{s+a-2}(u^{3})\right\|_{L_{xT}^{2}}\leq C\|u\|_{X_{s,b}}^{3}.\label{3.043}
\end{eqnarray}
According to the duality idea, to prove \eqref{3.043}, we only need to prove
\begin{eqnarray}
&&\left|\int_{0}^{T}\int_{\SR}|D|\langle D\rangle^{s+a-2}\left(\langle D\rangle^{-s} u_{1}\langle D\rangle^{-s} u_{2}\langle D\rangle^{-s} u_{3}\right)(x,t)
\overline{h}(x,t)dxdt\right|\nonumber\\
&&\leq C\|h\|_{L_{xT}^{2}}\prod_{j=1}^{3}\|u_{j}\|_{X_{0,b}}.\label{3.044}
\end{eqnarray}
We denote
\begin{eqnarray}
&&I_{4}:=
\int_{0}^{T}\int_{\SR}|D|\langle D\rangle^{s+a-2}\left(\langle D\rangle^{-s} u_{1}\langle D\rangle^{-s} u_{2}\langle D\rangle^{-s} u_{3}\right)(x,t)\overline{h}(x,t)dxdt.\label{3.045}
\end{eqnarray}
Let $P_{N}h$ and $P_{N_{j}}u_{j}(j=1,2,3)$ be the dyadic decompositions of $h$ and $u_{j}$, respectively.
Then, we have that $\supp \mathscr{F}_{x}P_{N}h\subset\{\xi\in\R:|\xi|\sim N\}$, and $\supp \mathscr{F}_{x}P_{N_{j}}u_{j}\subset \{\xi\in\R:|\xi|\sim N_{j}\}$, where $N$ and $N_{j}$ are dyadic numbers. We define $\sum=\sum\limits_{N_{1}, N_{2}, N_{3}, N}$. Without loss of generality, we assume $N_{1}\geq N_{2}\geq N_{3}$.

\noindent We consider the following cases: {\bf Case 1}: $N_{1}\leq 10$; {\bf Case
2}: $N_{1}\geq 10$, $N_{1}\gg N_{2}\geq N_{3}$ and {\bf Case 3}: $N_{1}\geq 10$, $N_{1}\sim N_{2}\gg N_{3}$; {\bf Case 4}: $N_{1}\geq 10$, $N_{1}\sim N_{2}\sim N_{3}\gg N$; {\bf Case 5}: $N_{1}\geq 10$, $N_{1}\sim N_{2}\sim N_{3}\sim N$.

\noindent {\bf Case 1}: When $N_{1}\leq 10$, by using H\"{o}lder inequality and Bernstein inequality, we have
\begin{eqnarray}
&&|I_{4}|\leq C\sum N\langle N\rangle^{s+a-2}\langle N_{1}\rangle^{-s}\langle N_{2}\rangle^{-s}\langle N_{3}\rangle^{-s}
\|P_{N_{1}}u_{1}\|_{L_{xT}^{2}}\|P_{N}h\|_{L_{xT}^{2}}\prod_{j=2}^{3}\|P_{N_{j}}u_{j}\|_{L_{xT}^{\infty}}\nonumber\\
&&\leq C\sum NN_{2}^{\frac{1}{2}}N_{3}^{\frac{1}{2}}\langle N_{1}\rangle^{-s}\langle N_{2}\rangle^{-s}
\|P_{N}h\|_{L_{xT}^{2}}\prod_{j=1}^{3}\|P_{N_{j}}u_{j}\|_{X_{0,b}}\nonumber\\
&&\leq C\sum NN_{1}^{\frac{1}{4}}N_{2}^{\frac{1}{4}}N_{3}^{\frac{1}{2}}\|P_{N}h\|_{L_{xT}^{2}}\prod_{j=1}^{3}\|P_{N_{j}}u_{j}\|_{X_{0,b}}\nonumber\\
&&\leq C\|h\|_{L_{xT}^{2}}\prod_{j=1}^{3}\|u_{j}\|_{X_{0,b}}.\label{3.046}
\end{eqnarray}
{\bf Case 2}: When $N_{1}\geq 10$,  $N_{1}\gg N_{2}\geq N_{3}$, we consider the following cases: {\bf Case 2A}: $N_{2}\leq 8$ and {\bf Case
2B}: $N_{2}\geq 8$.

\noindent {\bf Case 2A}: $N_{1}\geq 10$,  $N_{1}\gg N_{2}\geq N_{3}$, $N_{2}\leq 8$, we have $N_{1}\sim N$.
By using H\"{o}lder inequality and Bernstein inequality, $0\leq a\leq 1-2\epsilon$, we have
\begin{eqnarray}
&&|I_{4}|\leq C\sum NN^{s+a-2}N_{1}^{-s}\langle N_{2}\rangle^{-s}\langle N_{3}\rangle^{-s}
\|P_{N_{1}}u_{1}\|_{L_{xT}^{2}}\|P_{N}h\|_{L_{xT}^{2}}\prod_{j=2}^{3}\|P_{N_{j}}u_{j}\|_{L_{xT}^{\infty}}\nonumber\\
&&\leq C\sum N^{a-1}N_{2}^{\frac{1}{2}}N_{3}^{\frac{1}{2}}\langle N_{2}\rangle^{-s}\langle N_{3}\rangle^{-s}
\|P_{N}h\|_{L_{xT}^{2}}\prod_{j=1}^{3}\|P_{N_{j}}u_{j}\|_{X_{0,b}}\nonumber\\
&&\leq C\sum N_{1}^{a-1}N_{2}^{\frac{1}{2}}N_{3}^{\frac{1}{2}}\|P_{N}h\|_{L_{xT}^{2}}\prod_{j=1}^{3}\|P_{N_{j}}u_{j}\|_{X_{0,b}}\nonumber\\
&&\leq C\|h\|_{L_{xT}^{2}}\prod_{j=1}^{3}\|u_{j}\|_{X_{0,b}}.\label{3.047}
\end{eqnarray}
{\bf Case 2B}: $N_{1}\geq 10$,  $N_{1}\gg N_{2}$, $N_{2}\geq 8$,  we consider the following  cases: {\bf Case 2$B_{1}$}: $N_{3}\leq 4$ and {\bf Case
2$B_{2}$}: $N_{3}\geq 4$.

\noindent{\bf Case 2$B_{1}$}: $N_{3}\leq 4$, we have $N\sim N_{1}$.
By using H\"{o}lder inequality and \eqref{2.03}, $s\geq-\frac{1}{6}+\frac{\epsilon}{2}$, $0\leq a\leq \min\{1-2\epsilon,s+\frac{5}{4}-\epsilon\}$, we have
\begin{eqnarray}
&&|I_{4}|\leq C\sum NN^{s+a-2}N_{1}^{-s}N_{2}^{-s}\langle N_{3}\rangle^{-s}
\|P_{N_{1}}u_{1}P_{N_{2}}u_{2}P_{N_{3}}u_{3}\|_{L_{xT}^{2}}\|P_{N}h\|_{L_{xT}^{2}}\nonumber\\
&&\leq C\sum N_{1}^{-\epsilon}N_{1}^{a-1+\epsilon}N_{1}^{-\frac{1}{4}}N_{2}^{-s}N_{3}^{\frac{1}{2}}
\|P_{N}h\|_{L_{xT}^{2}}\prod_{j=1}^{3}\|P_{N_{j}}u_{j}\|_{X_{0,b}}\nonumber\\
&&\leq C\sum N_{1}^{-\epsilon} N_{2}^{-s+a-\frac{5}{4}+\epsilon}N_{3}^{\frac{1}{2}}\|P_{N}h\|_{L_{xT}^{2}}\prod_{j=1}^{3}\|P_{N_{j}}u_{j}\|_{X_{0,b}}\nonumber\\
&&\leq C\|h\|_{L_{xT}^{2}}\prod_{j=1}^{3}\|u_{j}\|_{X_{0,b}}.\label{3.048}
\end{eqnarray}
{\bf Case 2$B_{2}$}: $N_{3}\geq 4$, we have $N\sim N_{1}$, $N_{2}\geq N_{3}\geq4$.
By using H\"{o}lder inequality and \eqref{2.04}, $s\geq-\frac{1}{6}+\frac{\epsilon}{2}$, $0\leq a\leq \min\{1-2\epsilon,2s+\frac{3}{2}-\epsilon\}$, we have
\begin{eqnarray}
&&|I_{4}|\leq C\sum NN^{s+a-2}N_{1}^{-s}N_{2}^{-s}N_{3}^{-s}
\|P_{N_{1}}u_{1}P_{N_{2}}u_{2}P_{N_{3}}u_{3}\|_{L_{xT}^{2}}\|P_{N}h\|_{L_{xT}^{2}}\nonumber\\
&&\leq C\sum N_{1}^{-\epsilon}N_{1}^{a-1+\epsilon}N_{1}^{-\frac{1}{4}}N_{2}^{-\frac{1}{4}}N_{2}^{-s}N_{3}^{-s}
\|P_{N}h\|_{L_{xT}^{2}}\prod_{j=1}^{3}\|P_{N_{j}}u_{j}\|_{X_{0,b}}\nonumber\\
&&\leq C\sum N_{1}^{-\epsilon} N_{3}^{-2s+a-\frac{3}{2}+\epsilon}\|P_{N}h\|_{L_{xT}^{2}}\prod_{j=1}^{3}\|P_{N_{j}}u_{j}\|_{X_{0,b}}\nonumber\\
&&\leq C\|h\|_{L_{xT}^{2}}\prod_{j=1}^{3}\|u_{j}\|_{X_{0,b}}.\label{3.049}
\end{eqnarray}
{\bf Case 3}: $N_{1}\geq 10$, $N_{1}\sim N_{2}\gg N_{3}$, we consider the following cases: {\bf Case 3A}: $N_{1}\gg N$ and {\bf Case
3B}: $N_{1}\sim N$.

\noindent {\bf Case 3A}: $N_{1}\geq 10$, $N_{1}\sim N_{2}\gg N_{3}$, $N_{1}\gg N$, we consider the following cases: {\bf Case 3$A_{1}$}: $N_{3}\leq 4$ and {\bf Case
3$A_{2}$}: $N_{3}\geq 4$.

\noindent{\bf Case 3$A_{1}$}: $N_{3}\leq 4$, when $N\leq 4$, by using H\"{o}lder inequality and \eqref{2.05}, $s\geq-\frac{1}{6}+\frac{\epsilon}{2}$, we have
\begin{eqnarray}
&&|I_{4}|\leq C\sum N\langle N\rangle^{s+a-2}N_{1}^{-s}N_{2}^{-s}\langle N_{3}\rangle^{-s}
\|P_{N_{1}}u_{1}P_{N_{2}}u_{2}P_{N_{3}}u_{3}\|_{L_{xT}^{2}}\|P_{N}h\|_{L_{xT}^{2}}\nonumber\\
&&\leq C\sum NN_{1}^{-\epsilon}N_{1}^{-\frac{1}{3}-2s+\epsilon}N_{3}^{\frac{1}{3}}
\|P_{N}h\|_{L_{xT}^{2}}\prod_{j=1}^{3}\|P_{N_{j}}u_{j}\|_{X_{0,b}}\nonumber\\
&&\leq C\sum NN_{1}^{-\epsilon}N_{3}^{\frac{1}{3}}\|P_{N}h\|_{L_{xT}^{2}}\prod_{j=1}^{3}\|P_{N_{j}}u_{j}\|_{X_{0,b}}\nonumber\\
&&\leq C\|h\|_{L_{xT}^{2}}\prod_{j=1}^{3}\|u_{j}\|_{X_{0,b}}.\label{3.050}
\end{eqnarray}
When $N\geq 4$, $-\frac{1}{6}+\frac{\epsilon}{2}\leq s\leq 0$, by using H\"{o}lder inequality and \eqref{2.05}, $0\leq a\leq 1-2\epsilon$,  we have
\begin{eqnarray}
&&|I_{4}|\leq C\sum N\langle N\rangle^{s+a-2}N_{1}^{-s}N_{2}^{-s}\langle N_{3}\rangle^{-s}
\|P_{N_{1}}u_{1}P_{N_{2}}u_{2}P_{N_{3}}u_{3}\|_{L_{xT}^{2}}\|P_{N}h\|_{L_{xT}^{2}}\nonumber\\
&&\leq C\sum N_{1}^{-\epsilon}N^{s+a-1}N_{1}^{-\frac{1}{3}-2s+\epsilon}N_{3}^{\frac{1}{3}}\|P_{N}h\|_{L_{xT}^{2}}
\prod_{j=1}^{3}\|P_{N_{j}}u_{j}\|_{X_{0,b}}\nonumber\\
&&\leq C\sum N_{1}^{-\epsilon}N_{1}^{-\frac{1}{3}-2s+\epsilon}N_{3}^{\frac{1}{3}}\|P_{N}h\|_{L_{xT}^{2}}
\prod_{j=1}^{3}\|P_{N_{j}}u_{j}\|_{X_{0,b}}\nonumber\\
&&\leq C\sum N_{1}^{-\epsilon}N_{3}^{\frac{1}{3}}\|P_{N}h\|_{L_{xT}^{2}}\prod_{j=1}^{3}\|P_{N_{j}}u_{j}\|_{X_{0,b}}\nonumber\\
&&\leq C\|h\|_{L_{xT}^{2}}\prod_{j=1}^{3}\|u_{j}\|_{X_{0,b}}.\label{3.051}
\end{eqnarray}
When $N\geq 4$, $s\geq 0$, by using H\"{o}lder inequality and \eqref{2.05}, $0\leq a\leq 1-2\epsilon$,  we have
\begin{eqnarray}
&&|I_{4}|\leq C\sum NN^{s+a-2}N_{1}^{-s}N_{2}^{-s}\langle N_{3}\rangle^{-s}
\|P_{N_{1}}u_{1}P_{N_{2}}u_{2}P_{N_{3}}u_{3}\|_{L_{xT}^{2}}\|P_{N}h\|_{L_{xT}^{2}}\nonumber\\
&&\leq C\sum N_{1}^{-\epsilon}N^{a-1}N_{1}^{-\frac{1}{3}-s+\epsilon}N_{3}^{\frac{1}{3}}\|P_{N}h\|_{L_{xT}^{2}}
\prod_{j=1}^{3}\|P_{N_{j}}u_{j}\|_{X_{0,b}}\nonumber\\
&&\leq C\sum N_{1}^{-\epsilon}N_{1}^{-\frac{1}{3}-s+\epsilon}N_{3}^{\frac{1}{3}}\|P_{N}h\|_{L_{xT}^{2}}
\prod_{j=1}^{3}\|P_{N_{j}}u_{j}\|_{X_{0,b}}\nonumber\\
&&\leq C\sum N_{1}^{-\epsilon}N_{3}^{\frac{1}{3}}\|P_{N}h\|_{L_{xT}^{2}}\prod_{j=1}^{3}\|P_{N_{j}}u_{j}\|_{X_{0,b}}\nonumber\\
&&\leq C\|h\|_{L_{xT}^{2}}\prod_{j=1}^{3}\|u_{j}\|_{X_{0,b}}.\label{3.052}
\end{eqnarray}

\noindent{\bf Case 3$A_{2}$}: $N_{3}\geq 4$, when $N\leq 4$, by using H\"{o}lder inequality and \eqref{2.04}, $s\geq -\frac{1}{6}+\frac{\epsilon}{2}$, we have
\begin{eqnarray}
&&|I_{4}|\leq C\sum N\langle N\rangle^{s+a-2}N_{1}^{-s}N_{2}^{-s} N_{3}^{-s}
\|P_{N_{1}}u_{1}P_{N_{2}}u_{2}P_{N_{3}}u_{3}\|_{L_{xT}^{2}}\|P_{N}h\|_{L_{xT}^{2}}\nonumber\\
&&\leq C\sum NN_{1}^{-\epsilon}N_{1}^{-\frac{1}{2}-2s+\epsilon}N_{3}^{-s}\|P_{N}h\|_{L_{xT}^{2}}
\prod_{j=1}^{3}\|P_{N_{j}}u_{j}\|_{X_{0,b}}\nonumber\\
&&\leq C\sum NN_{1}^{-\epsilon}\|P_{N}h\|_{L_{xT}^{2}}\prod_{j=1}^{3}\|P_{N_{j}}u_{j}\|_{X_{0,b}}
\leq C\|h\|_{L_{xT}^{2}}\prod_{j=1}^{3}\|u_{j}\|_{X_{0,b}}.\label{3.053}
\end{eqnarray}
When $N\geq 4$, $-\frac{1}{6}+\frac{\epsilon}{2}\leq s\leq 0$, by using H\"{o}lder inequality and \eqref{2.04}, $0\leq a\leq 1-2\epsilon$,  we have
\begin{eqnarray}
&&|I_{4}|\leq C\sum NN^{s+a-2}N_{1}^{-s}N_{2}^{-s}N_{3}^{-s}
\|P_{N_{1}}u_{1}P_{N_{2}}u_{2}P_{N_{3}}u_{3}\|_{L_{xT}^{2}}\|P_{N}h\|_{L_{xT}^{2}}\nonumber\\
&&\leq C\sum N_{1}^{-\epsilon}N^{s+a-1}N_{1}^{-\frac{1}{2}-2s+\epsilon}N_{3}^{-s}\|P_{N}h\|_{L_{xT}^{2}}
\prod_{j=1}^{3}\|P_{N_{j}}u_{j}\|_{X_{0,b}}\nonumber\\
&&\leq C\sum N_{1}^{-\epsilon}N_{1}^{-\frac{1}{2}-3s+\epsilon}
\|P_{N}h\|_{L_{xT}^{2}}\prod_{j=1}^{3}\|P_{N_{j}}u_{j}\|_{X_{0,b}}\nonumber\\
&&\leq C\sum N_{1}^{-\epsilon}\|P_{N}h\|_{L_{xT}^{2}}\prod_{j=1}^{3}\|P_{N_{j}}u_{j}\|_{X_{0,b}}
\leq C\|h\|_{L_{xT}^{2}}\prod_{j=1}^{3}\|u_{j}\|_{X_{0,b}}.\label{3.054}
\end{eqnarray}
When $N\geq 4$, $s\geq 0$, by using H\"{o}lder inequality and \eqref{2.04}, $0\leq a\leq 1-2\epsilon$,  we have
\begin{eqnarray}
&&|I_{4}|\leq C\sum NN^{s+a-2}N_{1}^{-s}N_{2}^{-s}N_{3}^{-s}
\|P_{N_{1}}u_{1}P_{N_{2}}u_{2}P_{N_{3}}u_{3}\|_{L_{xT}^{2}}\|P_{N}h\|_{L_{xT}^{2}}\nonumber\\
&&\leq C\sum N_{1}^{-\epsilon}N^{s+a-1}N_{1}^{-\frac{1}{2}-2s+\epsilon}N_{3}^{-s}\|P_{N}h\|_{L_{xT}^{2}}
\prod_{j=1}^{3}\|P_{N_{j}}u_{j}\|_{X_{0,b}}\nonumber\\
&&\leq C\sum N_{1}^{-\epsilon}N_{1}^{-\frac{1}{2}-s+\epsilon}
\|P_{N}h\|_{L_{xT}^{2}}\prod_{j=1}^{3}\|P_{N_{j}}u_{j}\|_{X_{0,b}}\nonumber\\
&&\leq C\sum N_{1}^{-\epsilon}\|P_{N}h\|_{L_{xT}^{2}}\prod_{j=1}^{3}\|P_{N_{j}}u_{j}\|_{X_{0,b}}
\leq C\|h\|_{L_{xT}^{2}}\prod_{j=1}^{3}\|u_{j}\|_{X_{0,b}}.\label{3.055}
\end{eqnarray}
{\bf Case 3B}: $N_{1}\sim N_{2}\sim N$, when $N_{3}\leq 4$,
by using H\"{o}lder inequality and \eqref{2.03}, $s\geq-\frac{1}{6}+\frac{\epsilon}{2}$, $0\leq a\leq s+\frac{5}{4}-\epsilon$, we have
\begin{eqnarray}
&&|I_{4}|\leq C\sum NN^{s+a-2}N_{1}^{-s}N_{2}^{-s}\langle N_{3}\rangle^{-s}
\|P_{N_{1}}u_{1}P_{N_{2}}u_{2}P_{N_{3}}\|_{L_{xT}^{2}}\|P_{N}h\|_{L_{xT}^{2}}\nonumber\\
&&\leq C\sum N_{1}^{-\epsilon}N^{s+a-1}N_{1}^{-\frac{1}{4}-2s+\epsilon}N_{3}^{\frac{1}{2}}\|P_{N}h\|_{L_{xT}^{2}}
\prod_{j=1}^{3}\|P_{N_{j}}u_{j}\|_{X_{0,b}}\nonumber\\
&&\leq C\sum N_{1}^{-\epsilon}N_{1}^{a-s-1-\frac{1}{4}+\epsilon}N_{3}^{\frac{1}{2}}\|P_{N}h\|_{L_{xT}^{2}}\prod_{j=1}^{3}\|P_{N_{j}}u_{j}\|_{X_{0,b}}\nonumber\\
&&\leq C\sum N_{1}^{-\epsilon}N_{3}^{\frac{1}{2}}\|P_{N}h\|_{L_{xT}^{2}}\prod_{j=1}^{3}\|P_{N_{j}}u_{j}\|_{X_{0,b}}\nonumber\\
&&\leq C\|h\|_{L_{xT}^{2}}\prod_{j=1}^{3}\|u_{j}\|_{X_{0,b}}.\label{3.056}
\end{eqnarray}
When $N_{3}\geq 4$,
by using H\"{o}lder inequality and \eqref{2.04}, $s\geq-\frac{1}{6}+\frac{\epsilon}{2}$, $0\leq a\leq \{s+\frac{3}{2}-\epsilon, 2s+\frac{3}{2}-\epsilon\}$, we have
\begin{eqnarray}
&&|I_{4}|\leq C\sum NN^{s+a-2}N_{1}^{-s}N_{2}^{-s}N_{3}^{-s}
\|P_{N_{1}}u_{1}P_{N_{2}}u_{2}P_{N_{3}}\|_{L_{xT}^{2}}\|P_{N}h\|_{L_{xT}^{2}}\nonumber\\
&&\leq C\sum N_{1}^{-\epsilon}N^{s+a-1}N_{1}^{-\frac{1}{2}-2s+\epsilon}N_{3}^{-s}\|P_{N}h\|_{L_{xT}^{2}}
\prod_{j=1}^{3}\|P_{N_{j}}u_{j}\|_{X_{0,b}}\nonumber\\
&&\leq C\sum N_{1}^{-\epsilon}N_{1}^{a-s-1-\frac{1}{2}+\epsilon}N_{3}^{-s}\|P_{N}h\|_{L_{xT}^{2}}\prod_{j=1}^{3}\|P_{N_{j}}u_{j}\|_{X_{0,b}}\nonumber\\
&&\leq C\sum N_{1}^{-\epsilon}\|P_{N}h\|_{L_{xT}^{2}}\prod_{j=1}^{3}\|P_{N_{j}}u_{j}\|_{X_{0,b}}\nonumber\\
&&\leq C\|h\|_{L_{xT}^{2}}\prod_{j=1}^{3}\|u_{j}\|_{X_{0,b}}.\label{3.057}
\end{eqnarray}
{\bf Case 4}: $N_{1}\geq 10$, $N_{1}\sim N_{2}\sim N_{3}\gg N$, when $N\leq 4$, by using H\"{o}lder inequality and \eqref{2.04}, $s\geq-\frac{1}{6}+\frac{\epsilon}{2}$, we have
\begin{eqnarray}
&&|I_{4}|\leq C\sum N\langle N\rangle^{s+a-2}N_{1}^{-s}N_{2}^{-s}N_{3}^{-s}
\|P_{N_{1}}u_{1}P_{N_{2}}u_{2}P_{N_{3}}\|_{L_{xT}^{2}}\|P_{N}h\|_{L_{xT}^{2}}\nonumber\\
&&\leq C\sum NN_{1}^{-\epsilon}N_{1}^{-\frac{1}{2}-3s+\epsilon}\|P_{N}h\|_{L_{xT}^{2}}
\prod_{j=1}^{3}\|P_{N_{j}}u_{j}\|_{X_{0,b}}\nonumber\\
&&\leq C\sum NN_{1}^{-\epsilon}\|P_{N}h\|_{L_{xT}^{2}}\prod_{j=1}^{3}\|P_{N_{j}}u_{j}\|_{X_{0,b}}\nonumber\\
&&\leq C\|h\|_{L_{xT}^{2}}\prod_{j=1}^{3}\|u_{j}\|_{X_{0,b}}.\label{3.058}
\end{eqnarray}
When $N\geq 4$, by using H\"{o}lder inequality and \eqref{2.04}, $0\leq a\leq 1-2\epsilon$, $s\geq -\frac{1}{6}+\frac{\epsilon}{2}$, we have
\begin{eqnarray}
&&|I_{4}|\leq C\sum NN^{s+a-2}N_{1}^{-s}N_{2}^{-s}N_{3}^{-s}
\|P_{N_{1}}u_{1}P_{N_{2}}u_{2}P_{N_{3}}\|_{L_{xT}^{2}}\|P_{N}h\|_{L_{xT}^{2}}\nonumber\\
&&\leq C\sum N_{1}^{-\epsilon}N^{s+a-1}N_{1}^{-\frac{1}{2}-3s+\epsilon}\|P_{N}h\|_{L_{xT}^{2}}
\prod_{j=1}^{3}\|P_{N_{j}}u_{j}\|_{X_{0,b}}\nonumber\\
&&\leq C\sum N_{1}^{-\epsilon}N^{s}N_{1}^{-\frac{1}{2}-3s+\epsilon}\|P_{N}h\|_{L_{xT}^{2}}
\prod_{j=1}^{3}\|P_{N_{j}}u_{j}\|_{X_{0,b}}\nonumber\\
&&\leq C\sum N_{1}^{-\epsilon}\|P_{N}h\|_{L_{xT}^{2}}\prod_{j=1}^{3}\|P_{N_{j}}u_{j}\|_{X_{0,b}}\nonumber\\
&&\leq C\|h\|_{L_{xT}^{2}}\prod_{j=1}^{3}\|u_{j}\|_{X_{0,b}}.\label{3.059}
\end{eqnarray}
{\bf Case 5}: $N_{1}\geq 10$, $N_{1}\sim N_{2}\sim N_{3}\sim N$, by using H\"{o}lder inequality and \eqref{2.04}, $s\geq-\frac{1}{6}+\frac{\epsilon}{2}$, $0\leq a\leq  2s+\frac{3}{2}-\epsilon$,  we have
\begin{eqnarray}
&&|I_{4}|\leq C\sum NN^{s+a-2}N_{1}^{-s}N_{2}^{-s}N_{3}^{-s}
\|P_{N_{1}}u_{1}P_{N_{2}}u_{2}P_{N_{3}}\|_{L_{xT}^{2}}\|P_{N}h\|_{L_{xT}^{2}}\nonumber\\
&&\leq C\sum N_{1}^{-\epsilon}N_{1}^{a-1-\frac{1}{2}-2s+\epsilon}\|P_{N}h\|_{L_{xT}^{2}}
\prod_{j=1}^{3}\|P_{N_{j}}u_{j}\|_{X_{0,b}}\nonumber\\
&&\leq C\sum N_{1}^{-\epsilon}\|P_{N}h\|_{L_{xT}^{2}}\prod_{j=1}^{3}\|P_{N_{j}}u_{j}\|_{X_{0,b}}\nonumber\\
&&\leq C\|h\|_{L_{xT}^{2}}\prod_{j=1}^{3}\|u_{j}\|_{X_{0,b}}.\label{3.060}
\end{eqnarray}

This completes the proof of Lemma 3.4.

\begin{Lemma}\label{Lemma3.5}
Let $2\leq n\leq 4$, $0<T<1$, $0<\epsilon\ll1$, $s\geq\frac{2n-3}{6}+\frac{4\epsilon}{3}$, $b_{0}=-\frac{1}{2}+\frac{\epsilon}{12}$, $b=\frac{1}{2}+\frac{\epsilon}{24}$ and
\begin{align*}
0\leq a&\leq \min\{1-2\epsilon,s+\frac{5-n}{2}-4\epsilon,2s+\frac{5-2n}{2}-4\epsilon\}\\
&=\min\{1-2\epsilon,2s+\frac{5-2n}{2}-4\epsilon\}.
\end{align*}
Then, we have
\begin{eqnarray}
&&\left\|\eta\left(\frac{t}{T}\right)\mathscr{F}_{xt}^{-1}\left(\frac{|\xi|^{2}\mathscr{F}_{xt}(u^{3})}{2i\phi_{\pm}(\xi)}\right)\right\|_{X_{s+a,b_{0}}}\leq CT^{-b_{0}}\|u\|_{X_{s,b}}^{3}.\label{3.061}
\end{eqnarray}
\end{Lemma}
\noindent {\bf Proof.}
Inspired by \cite{ET2025}, we prove Lemma 3.5.
From \eqref{3.02}, to prove \eqref{3.061}, it suffices to show that
\begin{eqnarray}
&&\left\|\eta\left(\frac{t}{T}\right)|D|\langle D\rangle^{-2}(u^{3})\right\|_{X_{s+a,b_{0}}}\leq CT^{-b_{0}}\|u\|_{X_{s,b}}^{3}.\label{3.062}
\end{eqnarray}
By using \cite[Lemma 2.11]{T2006}, we have
\begin{eqnarray}
&&\left\|\eta\left(\frac{t}{T}\right)|D|\langle D\rangle^{-2}(u^{3})\right\|_{X_{s+a,b_{0}}}\leq CT^{-b_{0}}
\left\|\eta\left(\frac{t}{T}\right)|D|\langle D\rangle^{-2}(u^{3})\right\|_{X_{s+a,0}}\nonumber\\
&&=CT^{-b_{0}}\left\||D|\langle D\rangle^{-2}(u^{3})\right\|_{L_{T}^{2}H_{x}^{s+a}}.\label{3.063}
\end{eqnarray}
By using \eqref{3.063}, to prove \eqref{3.061}, it suffices to show that
\begin{eqnarray}
&&\left\||D|\langle D\rangle^{s+a-2}(u^{3})\right\|_{L_{xT}^{2}}\leq C\|u\|_{X_{s,b}}^{3}.\label{3.064}
\end{eqnarray}
According to the duality idea, to prove \eqref{3.064}, we only need to prove
\begin{eqnarray}
&&\left|\int_{0}^{T}\int_{\SR^{n}}|D|\langle D\rangle^{s+a-2}\left(\langle D\rangle^{-s} u_{1}\langle D\rangle^{-s} u_{2}\langle D\rangle^{-s} u_{3}\right)(x,t)
\overline{h}(x,t)dxdt\right|\nonumber\\
&&\leq C\|h\|_{L_{T}^{2}L_{x}}\prod_{j=1}^{3}\|u_{j}\|_{X_{0,b}}.\label{3.065}
\end{eqnarray}
We denote
\begin{eqnarray}
&&I_{5}:=
\int_{0}^{T}\int_{\SR^{n}}|D|\langle D\rangle^{s+a-2}\left(\langle D\rangle^{-s} u_{1}\langle D\rangle^{-s} u_{2}\langle D\rangle^{-s} u_{3}\right)(x,t)\overline{h}(x,t)dxdt.\label{3.066}
\end{eqnarray}
Let $P_{N}h$ and $P_{N_{j}}u_{j}(j=1,2,3)$ be the dyadic decompositions of $h$ and $u_{j}$, respectively.
Then, we have that $\supp \mathscr{F}_{x}P_{N}h\subset\{\xi\in\R^{n}:|\xi|\sim N\}$, and $\supp \mathscr{F}_{x}P_{N_{j}}u_{j}\subset \{\xi\in\R^{n}:|\xi|\sim N_{j}\}$, where $N$ and $N_{j}$ are dyadic numbers. We define $\sum=\sum\limits_{N_{1}, N_{2}, N_{3}, N}$. Without loss of generality, we assume $N_{1}\geq N_{2}\geq N_{3}$.

\noindent We consider the following cases: {\bf Case 1}: $N_{1}\leq 10$; {\bf Case
2}: $N_{1}\geq 10$, $N_{1}\gg N_{2}\geq N_{3}$ and {\bf Case 3}: $N_{1}\geq 10$, $N_{1}\sim N_{2}\gg N_{3}$; {\bf Case 4}: $N_{1}\geq 10$, $N_{1}\sim N_{2}\sim N_{3}\gg N$; {\bf Case 5}: $N_{1}\geq 10$, $N_{1}\sim N_{2}\sim N_{3}\sim N$.

\noindent {\bf Case 1}: When $N_{1}\leq 10$, by using H\"{o}lder inequality and Bernstein inequality, we have
\begin{eqnarray}
&&|I_{5}|\leq C\sum N\langle N\rangle^{s+a-2}\langle N_{1}\rangle^{-s}\langle N_{2}\rangle^{-s}\langle N_{3}\rangle^{-s}
\|P_{N_{1}}u_{1}\|_{L_{xT}^{2}}\|P_{N}h\|_{L_{xT}^{2}}\prod_{j=2}^{3}\|P_{N_{j}}u_{j}\|_{L_{xT}^{\infty}}\nonumber\\
&&\leq C\sum NN_{2}^{\frac{n}{2}}N_{3}^{\frac{n}{2}}\langle N_{1}\rangle^{-s}\langle N_{2}\rangle^{-s}
\|P_{N}h\|_{L_{xT}^{2}}\prod_{j=1}^{3}\|P_{N_{j}}u_{j}\|_{X_{0,b}}\nonumber\\
&&\leq C\sum NN_{1}^{\frac{n}{4}}N_{2}^{\frac{n}{4}}N_{3}^{\frac{n}{2}}\|P_{N}h\|_{L_{xT}^{2}}\prod_{j=1}^{3}\|P_{N_{j}}u_{j}\|_{X_{0,b}}\nonumber\\
&&\leq C\|h\|_{L_{xT}^{2}}\prod_{j=1}^{3}\|u_{j}\|_{X_{0,b}}.\label{3.067}
\end{eqnarray}
{\bf Case 2}: When $N_{1}\geq 10$,  $N_{1}\gg N_{2}\geq N_{3}$, we consider the following cases: {\bf Case 2A}: $N_{2}\leq 8$ and {\bf Case
2B}: $N_{2}\geq 8$.

\noindent {\bf Case 2A}: $N_{1}\geq 10$,  $N_{1}\gg N_{2}\geq N_{3}$, $N_{2}\leq 8$, we have $N_{1}\sim N$.
By using H\"{o}lder inequality and Bernstein inequality, $0\leq a\leq 1-2\epsilon$, we have
\begin{eqnarray}
&&|I_{5}|\leq C\sum NN^{s+a-2}N_{1}^{-s}\langle N_{2}\rangle^{-s}\langle N_{3}\rangle^{-s}
\|P_{N_{1}}u_{1}\|_{L_{xT}^{2}}\|P_{N}h\|_{L_{xT}^{2}}\prod_{j=2}^{3}\|P_{N_{j}}u_{j}\|_{L_{xT}^{\infty}}\nonumber\\
&&\leq C\sum N^{a-1}N_{2}^{\frac{n}{2}}N_{3}^{\frac{n}{2}}\langle N_{2}\rangle^{-s}\langle N_{3}\rangle^{-s}
\|P_{N}h\|_{L_{xT}^{2}}\prod_{j=1}^{3}\|P_{N_{j}}u_{j}\|_{X_{0,b}}\nonumber\\
&&\leq C\sum N_{1}^{a-1}N_{2}^{\frac{n}{2}}N_{3}^{\frac{n}{2}}\|P_{N}h\|_{L_{xT}^{2}}\prod_{j=1}^{3}\|P_{N_{j}}u_{j}\|_{X_{0,b}}\nonumber\\
&&\leq C\|h\|_{L_{xT}^{2}}\prod_{j=1}^{3}\|u_{j}\|_{X_{0,b}}.\label{3.068}
\end{eqnarray}
{\bf Case 2B}: $N_{1}\geq 10$,  $N_{1}\gg N_{2}$, $N_{2}\geq 8$,  we consider the following  cases: {\bf Case 2$B_{1}$}: $N_{3}\leq 4$ and {\bf Case
2$B_{2}$}: $N_{3}\geq 4$.

\noindent{\bf Case 2$B_{1}$}: $N_{3}\leq 4$, we have $N\sim N_{1}$.
By using H\"{o}lder inequality and \eqref{2.06}, $2\leq n\leq 4$, $s\geq\frac{2n-3}{6}+2\epsilon$, $0\leq a\leq \min\{1-2\epsilon,s+\frac{5-n}{2}-4\epsilon\}$, we have
\begin{eqnarray}
&&|I_{5}|\leq C\sum NN^{s+a-2}N_{1}^{-s}N_{2}^{-s}\langle N_{3}\rangle^{-s}
\|P_{N_{1}}u_{1}P_{N_{2}}u_{2}P_{N_{3}}u_{3}\|_{L_{xT}^{2}}\|P_{N}h\|_{L_{xT}^{2}}\nonumber\\
&&\leq C\sum N_{1}^{-\epsilon}N_{1}^{a-1+\epsilon}N_{1}^{-\frac{1}{2}+\epsilon}N_{2}^{\frac{n}{2}-1-s+2\epsilon}N_{3}^{\frac{n}{2}}
\|P_{N}h\|_{L_{xT}^{2}}\prod_{j=1}^{3}\|P_{N_{j}}u_{j}\|_{X_{0,b}}\nonumber\\
&&\leq C\sum N_{1}^{-\epsilon} N_{2}^{-s+a-\frac{5-n}{2}+4\epsilon}N_{3}^{\frac{n}{2}}\|P_{N}h\|_{L_{xT}^{2}}\prod_{j=1}^{3}\|P_{N_{j}}u_{j}\|_{X_{0,b}}\nonumber\\
&&\leq C\|h\|_{L_{xT}^{2}}\prod_{j=1}^{3}\|u_{j}\|_{X_{0,b}}.\label{3.069}
\end{eqnarray}
{\bf Case 2$B_{2}$}: $N_{3}\geq 4$, we have $N\sim N_{1}$, $N_{2}\geq N_{3}\geq4$. When $s\geq \frac{n}{2}$,
by using H\"{o}lder inequality and \eqref{2.06}, $0\leq a\leq 1-2\epsilon$, we have
\begin{eqnarray}
&&|I_{5}|\leq C\sum NN^{s+a-2}N_{1}^{-s}N_{2}^{-s}N_{3}^{-s}
\|P_{N_{1}}u_{1}P_{N_{2}}u_{2}P_{N_{3}}u_{3}\|_{L_{xT}^{2}}\|P_{N}h\|_{L_{xT}^{2}}\nonumber\\
&&\leq C\sum N_{1}^{-\epsilon}N_{1}^{a-1+\epsilon}N_{1}^{-\frac{1}{2}+\epsilon}N_{2}^{\frac{n}{2}-1+2\epsilon}N_{2}^{-s}N_{3}^{-s+\frac{n}{2}}
\|P_{N}h\|_{L_{xT}^{2}}\prod_{j=1}^{3}\|P_{N_{j}}u_{j}\|_{X_{0,b}}\nonumber\\
&&\leq C\sum N_{1}^{-\epsilon} \|P_{N}h\|_{L_{xT}^{2}}\prod_{j=1}^{3}\|P_{N_{j}}u_{j}\|_{X_{0,b}}\nonumber\\
&&\leq C\|h\|_{L_{xT}^{2}}\prod_{j=1}^{3}\|u_{j}\|_{X_{0,b}}.\label{3.070}
\end{eqnarray}
When $s\leq \frac{n}{2}$,
by using H\"{o}lder inequality and \eqref{2.06}, $2\leq n\leq 4$, $s\geq\frac{2n-3}{6}+2\epsilon$, $0\leq a\leq \min\{1-2\epsilon,2s+\frac{5-2n}{2}-4\epsilon\}$, we have
\begin{eqnarray}
&&|I_{5}|\leq C\sum NN^{s+a-2}N_{1}^{-s}N_{2}^{-s}N_{3}^{-s}
\|P_{N_{1}}u_{1}P_{N_{2}}u_{2}P_{N_{3}}u_{3}\|_{L_{xT}^{2}}\|P_{N}h\|_{L_{xT}^{2}}\nonumber\\
&&\leq C\sum N_{1}^{-\epsilon}N_{1}^{a-1+\epsilon}N_{1}^{-\frac{1}{2}+\epsilon}N_{2}^{\frac{n}{2}-1+2\epsilon}N_{2}^{-s}N_{3}^{-s+\frac{n}{2}}
\|P_{N}h\|_{L_{xT}^{2}}\prod_{j=1}^{3}\|P_{N_{j}}u_{j}\|_{X_{0,b}}\nonumber\\
&&\leq C\sum N_{1}^{-\epsilon} N_{3}^{a-2s-\frac{5-2n}{2}+4\epsilon}\|P_{N}h\|_{L_{xT}^{2}}\prod_{j=1}^{3}\|P_{N_{j}}u_{j}\|_{X_{0,b}}\nonumber\\
&&\leq C\|h\|_{L_{xT}^{2}}\prod_{j=1}^{3}\|u_{j}\|_{X_{0,b}}.\label{3.071}
\end{eqnarray}
{\bf Case 3}: $N_{1}\geq 10$, $N_{1}\sim N_{2}\gg N_{3}$, we consider the following cases: {\bf Case 3A}: $N_{1}\gg N$ and {\bf Case
3B}: $N_{1}\sim N$.

\noindent {\bf Case 3A}: $N_{1}\geq 10$, $N_{1}\sim N_{2}\gg N_{3}$, $N_{1}\gg N$, we consider the following cases: {\bf Case 3$A_{1}$}: $N_{3}\leq 4$ and {\bf Case
3$A_{2}$}: $N_{3}\geq 4$.

\noindent{\bf Case 3$A_{1}$}: $N_{3}\leq 4$, when $N\leq 4$, by using H\"{o}lder inequality and \eqref{2.06}, $s\geq \frac{n-3}{4}+2\epsilon$, we have
\begin{eqnarray}
&&|I_{5}|\leq C\sum N\langle N\rangle^{s+a-2}N_{1}^{-s}N_{2}^{-s}\langle N_{3}\rangle^{-s}
\|P_{N_{1}}u_{1}P_{N_{2}}u_{2}P_{N_{3}}u_{3}\|_{L_{xT}^{2}}\|P_{N}h\|_{L_{xT}^{2}}\nonumber\\
&&\leq C\sum NN_{1}^{-\epsilon}N_{1}^{\frac{n-3}{2}-2s+4\epsilon}N_{3}^{\frac{n}{2}}
\|P_{N}h\|_{L_{xT}^{2}}\prod_{j=1}^{3}\|P_{N_{j}}u_{j}\|_{X_{0,b}}\nonumber\\
&&\leq C\sum NN_{1}^{-\epsilon}N_{3}^{\frac{n}{2}}\|P_{N}h\|_{L_{xT}^{2}}\prod_{j=1}^{3}\|P_{N_{j}}u_{j}\|_{X_{0,b}}\nonumber\\
&&\leq C\|h\|_{L_{xT}^{2}}\prod_{j=1}^{3}\|u_{j}\|_{X_{0,b}}.\label{3.072}
\end{eqnarray}
When $N\geq 4$, $n=2$, $\frac{n-3}{4}+2\epsilon\leq s\leq 0$, by using H\"{o}lder inequality and \eqref{2.06}, $0\leq a\leq 1-2\epsilon$,  we have
\begin{eqnarray}
&&|I_{5}|\leq C\sum NN^{s+a-2}N_{1}^{-s}N_{2}^{-s}\langle N_{3}\rangle^{-s}
\|P_{N_{1}}u_{1}P_{N_{2}}u_{2}P_{N_{3}}u_{3}\|_{L_{xT}^{2}}\|P_{N}h\|_{L_{xT}^{2}}\nonumber\\
&&\leq C\sum N_{1}^{-\epsilon}N^{s+a-1}N_{1}^{\frac{n-3}{2}-2s+4\epsilon}N_{3}^{\frac{n}{2}}\|P_{N}h\|_{L_{xT}^{2}}
\prod_{j=1}^{3}\|P_{N_{j}}u_{j}\|_{X_{0,b}}\nonumber\\
&&\leq C\sum N_{1}^{-\epsilon}N_{1}^{\frac{n-3}{2}-2s+4\epsilon}N_{3}^{\frac{n}{2}}\|P_{N}h\|_{L_{xT}^{2}}
\prod_{j=1}^{3}\|P_{N_{j}}u_{j}\|_{X_{0,b}}\nonumber\\
&&\leq C\sum N_{1}^{-\epsilon}N_{3}^{\frac{n}{2}}\|P_{N}h\|_{L_{xT}^{2}}\prod_{j=1}^{3}\|P_{N_{j}}u_{j}\|_{X_{0,b}}\nonumber\\
&&\leq C\|h\|_{L_{xT}^{2}}\prod_{j=1}^{3}\|u_{j}\|_{X_{0,b}}.\label{3.073}
\end{eqnarray}
When $N\geq 4$, $n=2$, $s\geq 0$, by using H\"{o}lder inequality and \eqref{2.06}, $0\leq a\leq 1-2\epsilon$,  we have
\begin{eqnarray}
&&|I_{5}|\leq C\sum NN^{s+a-2}N_{1}^{-s}N_{2}^{-s}\langle N_{3}\rangle^{-s}
\|P_{N_{1}}u_{1}P_{N_{2}}u_{2}P_{N_{3}}u_{3}\|_{L_{xT}^{2}}\|P_{N}h\|_{L_{xT}^{2}}\nonumber\\
&&\leq C\sum N_{1}^{-\epsilon}N^{s+a-1}N_{1}^{\frac{n-3}{2}-2s+4\epsilon}N_{3}^{\frac{n}{2}}\|P_{N}h\|_{L_{xT}^{2}}
\prod_{j=1}^{3}\|P_{N_{j}}u_{j}\|_{X_{0,b}}\nonumber\\
&&\leq C\sum N_{1}^{-\epsilon}N_{1}^{\frac{n-3}{2}-s+2\epsilon}N_{3}^{\frac{n}{2}}\|P_{N}h\|_{L_{xT}^{2}}
\prod_{j=1}^{3}\|P_{N_{j}}u_{j}\|_{X_{0,b}}\nonumber\\
&&\leq C\sum N_{1}^{-\epsilon}N_{3}^{\frac{n}{2}}\|P_{N}h\|_{L_{xT}^{2}}\prod_{j=1}^{3}\|P_{N_{j}}u_{j}\|_{X_{0,b}}\nonumber\\
&&\leq C\|h\|_{L_{xT}^{2}}\prod_{j=1}^{3}\|u_{j}\|_{X_{0,b}}.\label{3.074}
\end{eqnarray}
When $N\geq 4$, $n=3,4$, $s\geq \frac{n-3}{2}+4\epsilon$, by using H\"{o}lder inequality and \eqref{2.06}, $0\leq a\leq 1-2\epsilon$,  we have
\begin{eqnarray}
&&|I_{5}|\leq C\sum NN^{s+a-2}N_{1}^{-s}N_{2}^{-s}\langle N_{3}\rangle^{-s}
\|P_{N_{1}}u_{1}P_{N_{2}}u_{2}P_{N_{3}}u_{3}\|_{L_{xT}^{2}}\|P_{N}h\|_{L_{xT}^{2}}\nonumber\\
&&\leq C\sum N_{1}^{-\epsilon}N^{s+a-1}N_{1}^{\frac{n-3}{2}-2s+4\epsilon}N_{3}^{\frac{n}{2}}\|P_{N}h\|_{L_{xT}^{2}}
\prod_{j=1}^{3}\|P_{N_{j}}u_{j}\|_{X_{0,b}}\nonumber\\
&&\leq C\sum N_{1}^{-\epsilon}N_{1}^{\frac{n-3}{2}-s+4\epsilon}N_{3}^{\frac{n}{2}}\|P_{N}h\|_{L_{xT}^{2}}
\prod_{j=1}^{3}\|P_{N_{j}}u_{j}\|_{X_{0,b}}\nonumber\\
&&\leq C\sum N_{1}^{-\epsilon}N_{3}^{\frac{n}{2}}\|P_{N}h\|_{L_{xT}^{2}}\prod_{j=1}^{3}\|P_{N_{j}}u_{j}\|_{X_{0,b}}\nonumber\\
&&\leq C\|h\|_{L_{xT}^{2}}\prod_{j=1}^{3}\|u_{j}\|_{X_{0,b}}.\label{3.075}
\end{eqnarray}
\noindent{\bf Case 3$A_{2}$}: $N_{3}\geq 4$, when $N\leq 4$, by using H\"{o}lder inequality and \eqref{2.06}, $s\geq \frac{2n-3}{6}+2\epsilon$, we have
\begin{eqnarray}
&&|I_{5}|\leq C\sum N\langle N\rangle^{s+a-2}N_{1}^{-s}N_{2}^{-s} N_{3}^{-s}
\|P_{N_{1}}u_{1}P_{N_{2}}u_{2}P_{N_{3}}u_{3}\|_{L_{xT}^{2}}\|P_{N}h\|_{L_{xT}^{2}}\nonumber\\
&&\leq C\sum NN_{1}^{-\epsilon}N_{1}^{\frac{n-3}{2}-2s+4\epsilon}N_{3}^{-s+\frac{n}{2}}\|P_{N}h\|_{L_{xT}^{2}}
\prod_{j=1}^{3}\|P_{N_{j}}u_{j}\|_{X_{0,b}}\nonumber\\
&&\leq C\sum NN_{1}^{-\epsilon}\|P_{N}h\|_{L_{xT}^{2}}\prod_{j=1}^{3}\|P_{N_{j}}u_{j}\|_{X_{0,b}}
\leq C\|h\|_{L_{xT}^{2}}\prod_{j=1}^{3}\|u_{j}\|_{X_{0,b}}.\label{3.076}
\end{eqnarray}
When $N\geq 4$, $s\geq\frac{n}{2}$, by using H\"{o}lder inequality and \eqref{2.06}, $0\leq a\leq 1-2\epsilon$, we have
\begin{eqnarray}
&&|I_{5}|\leq C\sum NN^{s+a-2}N_{1}^{-s}N_{2}^{-s}N_{3}^{-s}
\|P_{N_{1}}u_{1}P_{N_{2}}u_{2}P_{N_{3}}\|_{L_{xT}^{2}}\|P_{N}h\|_{L_{xT}^{2}}\nonumber\\
&&\leq C\sum N_{1}^{-\epsilon}N^{s+a-1}N_{1}^{\frac{n-3}{2}-2s+4\epsilon}N_{3}^{-s+\frac{n}{2}}\|P_{N}h\|_{L_{xT}^{2}}
\prod_{j=1}^{3}\|P_{N_{j}}u_{j}\|_{X_{0,b}}\nonumber\\
&&\leq C\sum N_{1}^{-\epsilon}\|P_{N}h\|_{L_{xT}^{2}}\prod_{j=1}^{3}\|P_{N_{j}}u_{j}\|_{X_{0,b}}\nonumber\\
&&\leq C\|h\|_{L_{xT}^{2}}\prod_{j=1}^{3}\|u_{j}\|_{X_{0,b}}.\label{3.077}
\end{eqnarray}
When $N\geq 4$, $2\leq n\leq 4$, $\frac{2n-3}{6}+2\epsilon\leq s\leq \frac{n}{2}$, $0\leq a\leq 1-s$, by using H\"{o}lder inequality and \eqref{2.06}, we have
\begin{eqnarray}
&&|I_{5}|\leq C\sum NN^{s+a-2}N_{1}^{-s}N_{2}^{-s} N_{3}^{-s}
\|P_{N_{1}}u_{1}P_{N_{2}}u_{2}P_{N_{3}}\|_{L_{xT}^{2}}\|P_{N}h\|_{L_{xT}^{2}}\nonumber\\
&&\leq C\sum N_{1}^{-\epsilon}N^{s+a-1}N_{1}^{\frac{n-3}{2}-2s+4\epsilon}N_{3}^{-s+\frac{n}{2}}\|P_{N}h\|_{L_{xT}^{2}}
\prod_{j=1}^{3}\|P_{N_{j}}u_{j}\|_{X_{0,b}}\nonumber\\
&&\leq C\sum N_{1}^{-\epsilon}N_{1}^{\frac{2n-3}{2}-3s+4\epsilon}\|P_{N}h\|_{L_{xT}^{2}}\prod_{j=1}^{3}\|P_{N_{j}}u_{j}\|_{X_{0,b}}\nonumber\\
&&\leq C\sum N_{1}^{-\epsilon}\|P_{N}h\|_{L_{xT}^{2}}\prod_{j=1}^{3}\|P_{N_{j}}u_{j}\|_{X_{0,b}}\nonumber\\
&&\leq C\|h\|_{L_{xT}^{2}}\prod_{j=1}^{3}\|u_{j}\|_{X_{0,b}}.\label{3.078}
\end{eqnarray}
When $N\geq 4$, $2\leq n\leq 4$, $\frac{2n-3}{6}+2\epsilon\leq s\leq \frac{n}{2}$, $1-s\leq a\leq 2s+\frac{5-2n}{2}-4\epsilon$, by using H\"{o}lder inequality and \eqref{2.06}, we have
\begin{eqnarray}
&&|I_{5}|\leq C\sum NN^{s+a-2}N_{1}^{-s}N_{2}^{-s}N_{3}^{-s}
\|P_{N_{1}}u_{1}P_{N_{2}}u_{2}P_{N_{3}}\|_{L_{xT}^{2}}\|P_{N}h\|_{L_{xT}^{2}}\nonumber\\
&&\leq C\sum N_{1}^{-\epsilon}N^{s+a-1}N_{1}^{\frac{n-3}{2}-2s+4\epsilon}N_{3}^{-s+\frac{n}{2}}\|P_{N}h\|_{L_{xT}^{2}}
\prod_{j=1}^{3}\|P_{N_{j}}u_{j}\|_{X_{0,b}}\nonumber\\
&&\leq C\sum N_{1}^{-\epsilon}N_{1}^{a-\frac{5-2n}{2}-2s+4\epsilon}\|P_{N}h\|_{L_{xT}^{2}}\prod_{j=1}^{3}\|P_{N_{j}}u_{j}\|_{X_{0,b}}\nonumber\\
&&\leq C\sum N_{1}^{-\epsilon}\|P_{N}h\|_{L_{xT}^{2}}\prod_{j=1}^{3}\|P_{N_{j}}u_{j}\|_{X_{0,b}}\nonumber\\
&&\leq C\|h\|_{L_{xT}^{2}}\prod_{j=1}^{3}\|u_{j}\|_{X_{0,b}}.\label{3.079}
\end{eqnarray}
{\bf Case 3B}: $N_{1}\sim N_{2}\sim N$, when $N_{3}\leq 4$,
by using H\"{o}lder inequality and \eqref{2.06}, $2\leq n\leq 4$, $s\geq\frac{2n-3}{6}+2\epsilon$, $0\leq a\leq s+\frac{5-n}{2}-4\epsilon$, we have
\begin{eqnarray}
&&|I_{5}|\leq C\sum NN^{s+a-2}N_{1}^{-s}N_{2}^{-s}\langle N_{3}\rangle^{-s}
\|P_{N_{1}}u_{1}P_{N_{2}}u_{2}P_{N_{3}}\|_{L_{xT}^{2}}\|P_{N}h\|_{L_{xT}^{2}}\nonumber\\
&&\leq C\sum N_{1}^{-\epsilon}N^{s+a-1}N_{1}^{\frac{n-3}{2}-2s+4\epsilon}N_{3}^{\frac{n}{2}}\|P_{N}h\|_{L_{xT}^{2}}
\prod_{j=1}^{3}\|P_{N_{j}}u_{j}\|_{X_{0,b}}\nonumber\\
&&\leq C\sum N_{1}^{-\epsilon}N_{1}^{a-s-\frac{5-n}{2}+4\epsilon}N_{3}^{\frac{n}{2}}\|P_{N}h\|_{L_{xT}^{2}}\prod_{j=1}^{3}\|P_{N_{j}}u_{j}\|_{X_{0,b}}\nonumber\\
&&\leq C\sum N_{1}^{-\epsilon}N_{3}^{\frac{n}{2}}\|P_{N}h\|_{L_{xT}^{2}}\prod_{j=1}^{3}\|P_{N_{j}}u_{j}\|_{X_{0,b}}\nonumber\\
&&\leq C\|h\|_{L_{xT}^{2}}\prod_{j=1}^{3}\|u_{j}\|_{X_{0,b}}.\label{3.080}
\end{eqnarray}
When $N_{3}\geq 4$, $s\leq \frac{n}{2}$,
by using H\"{o}lder inequality and \eqref{2.06}, $2\leq n\leq 4$, $s\geq\frac{2n-3}{6}+2\epsilon$, $0\leq a\leq  2s+\frac{5-2n}{2}-4\epsilon$, we have
\begin{eqnarray}
&&|I_{5}|\leq C\sum NN^{s+a-2}N_{1}^{-s}N_{2}^{-s}N_{3}^{-s}
\|P_{N_{1}}u_{1}P_{N_{2}}u_{2}P_{N_{3}}\|_{L_{xT}^{2}}\|P_{N}h\|_{L_{xT}^{2}}\nonumber\\
&&\leq C\sum N_{1}^{-\epsilon}N^{s+a-1}N_{1}^{\frac{n-3}{2}-2s+4\epsilon}N_{3}^{-s+\frac{n}{2}}\|P_{N}h\|_{L_{xT}^{2}}
\prod_{j=1}^{3}\|P_{N_{j}}u_{j}\|_{X_{0,b}}\nonumber\\
&&\leq C\sum N_{1}^{-\epsilon}N_{1}^{a-2s-\frac{5-2n}{2}+4\epsilon}\|P_{N}h\|_{L_{xT}^{2}}\prod_{j=1}^{3}\|P_{N_{j}}u_{j}\|_{X_{0,b}}\nonumber\\
&&\leq C\sum N_{1}^{-\epsilon}\|P_{N}h\|_{L_{xT}^{2}}\prod_{j=1}^{3}\|P_{N_{j}}u_{j}\|_{X_{0,b}}\nonumber\\
&&\leq C\|h\|_{L_{xT}^{2}}\prod_{j=1}^{3}\|u_{j}\|_{X_{0,b}}.\label{3.081}
\end{eqnarray}
When $N_{3}\geq 4$, $s\geq \frac{n}{2}$,
by using H\"{o}lder inequality and \eqref{2.06}, $2\leq n\leq 4$, $s\geq\frac{2n-3}{6}+2\epsilon$, $0\leq a\leq  s+\frac{5-n}{2}-4\epsilon$, we have
\begin{eqnarray}
&&|I_{5}|\leq C\sum NN^{s+a-2}N_{1}^{-s}N_{2}^{-s}N_{3}^{-s}
\|P_{N_{1}}u_{1}P_{N_{2}}u_{2}P_{N_{3}}\|_{L_{xT}^{2}}\|P_{N}h\|_{L_{xT}^{2}}\nonumber\\
&&\leq C\sum N_{1}^{-\epsilon}N^{s+a-1}N_{1}^{\frac{n-3}{2}-2s+4\epsilon}N_{3}^{-s+\frac{n}{2}}\|P_{N}h\|_{L_{xT}^{2}}
\prod_{j=1}^{3}\|P_{N_{j}}u_{j}\|_{X_{0,b}}\nonumber\\
&&\leq C\sum N_{1}^{-\epsilon}N_{1}^{a-s-\frac{5-n}{2}+4\epsilon}\|P_{N}h\|_{L_{xT}^{2}}\prod_{j=1}^{3}\|P_{N_{j}}u_{j}\|_{X_{0,b}}\nonumber\\
&&\leq C\sum N_{1}^{-\epsilon}\|P_{N}h\|_{L_{xT}^{2}}\prod_{j=1}^{3}\|P_{N_{j}}u_{j}\|_{X_{0,b}}\nonumber\\
&&\leq C\|h\|_{L_{xT}^{2}}\prod_{j=1}^{3}\|u_{j}\|_{X_{0,b}}.\label{3.082}
\end{eqnarray}
{\bf Case 4}: $N_{1}\geq 10$, $N_{1}\sim N_{2}\sim N_{3}\gg N$, when $N\leq 4$, by using H\"{o}lder inequality and \eqref{2.06}, $s\geq\frac{2n-3}{6}+2\epsilon$, we have
\begin{eqnarray}
&&|I_{5}|\leq C\sum N\langle N\rangle^{s+a-2}N_{1}^{-s}N_{2}^{-s}N_{3}^{-s}
\|P_{N_{1}}u_{1}P_{N_{2}}u_{2}P_{N_{3}}\|_{L_{xT}^{2}}\|P_{N}h\|_{L_{xT}^{2}}\nonumber\\
&&\leq C\sum NN_{1}^{-\epsilon}N_{1}^{\frac{2n-3}{2}-3s+4\epsilon}\|P_{N}h\|_{L_{xT}^{2}}
\prod_{j=1}^{3}\|P_{N_{j}}u_{j}\|_{X_{0,b}}\nonumber\\
&&\leq C\sum NN_{1}^{-\epsilon}\|P_{N}h\|_{L_{xT}^{2}}\prod_{j=1}^{3}\|P_{N_{j}}u_{j}\|_{X_{0,b}}\nonumber\\
&&\leq C\|h\|_{L_{xT}^{2}}\prod_{j=1}^{3}\|u_{j}\|_{X_{0,b}}.\label{3.083}
\end{eqnarray}
When $N\geq 4$, $0\leq a\leq 1-s$, by using H\"{o}lder inequality and \eqref{2.06}, $s\geq \frac{2n-3}{6}+2\epsilon$, we have
\begin{eqnarray}
&&|I_{5}|\leq C\sum NN^{s+a-2}N_{1}^{-s}N_{2}^{-s}N_{3}^{-s}
\|P_{N_{1}}u_{1}P_{N_{2}}u_{2}P_{N_{3}}\|_{L_{xT}^{2}}\|P_{N}h\|_{L_{xT}^{2}}\nonumber\\
&&\leq C\sum N_{1}^{-\epsilon}N^{s+a-1}N_{1}^{\frac{2n-3}{2}-3s+4\epsilon}\|P_{N}h\|_{L_{xT}^{2}}
\prod_{j=1}^{3}\|P_{N_{j}}u_{j}\|_{X_{0,b}}\nonumber\\
&&\leq C\sum N_{1}^{-\epsilon}N_{1}^{\frac{2n-3}{2}-3s+4\epsilon}\|P_{N}h\|_{L_{xT}^{2}}
\prod_{j=1}^{3}\|P_{N_{j}}u_{j}\|_{X_{0,b}}\nonumber\\
&&\leq C\sum N_{1}^{-\epsilon}\|P_{N}h\|_{L_{xT}^{2}}\prod_{j=1}^{3}\|P_{N_{j}}u_{j}\|_{X_{0,b}}\nonumber\\
&&\leq C\|h\|_{L_{xT}^{2}}\prod_{j=1}^{3}\|u_{j}\|_{X_{0,b}}.\label{3.084}
\end{eqnarray}
When $N\geq 4$, $s\geq\frac{2n-3}{6}+2\epsilon$, $1-s\leq a\leq 2s+\frac{5-2n}{2}-4\epsilon$, by using H\"{o}lder inequality and \eqref{2.06}, we have
\begin{eqnarray}
&&|I_{5}|\leq C\sum NN^{s+a-2}N_{1}^{-s}N_{2}^{-s}N_{3}^{-s}
\|P_{N_{1}}u_{1}P_{N_{2}}u_{2}P_{N_{3}}\|_{L_{xT}^{2}}\|P_{N}h\|_{L_{xT}^{2}}\nonumber\\
&&\leq C\sum N_{1}^{-\epsilon}N^{s+a-1}N_{1}^{\frac{2n-3}{2}-3s+4\epsilon}\|P_{N}h\|_{L_{xT}^{2}}
\prod_{j=1}^{3}\|P_{N_{j}}u_{j}\|_{X_{0,b}}\nonumber\\
&&\leq C\sum N_{1}^{-\epsilon}N_{1}^{a-\frac{5-2n}{2}-2s+4\epsilon}\|P_{N}h\|_{L_{xT}^{2}}
\prod_{j=1}^{3}\|P_{N_{j}}u_{j}\|_{X_{0,b}}\nonumber\\
&&\leq C\sum N_{1}^{-\epsilon}\|P_{N}h\|_{L_{xT}^{2}}\prod_{j=1}^{3}\|P_{N_{j}}u_{j}\|_{X_{0,b}}\nonumber\\
&&\leq C\|h\|_{L_{xT}^{2}}\prod_{j=1}^{3}\|u_{j}\|_{X_{0,b}}.\label{3.085}
\end{eqnarray}
{\bf Case 5}: $N_{1}\geq 10$, $N_{1}\sim N_{2}\sim N_{3}\sim N$, by using H\"{o}lder inequality and \eqref{2.06}, $2\leq n\leq 4$, $s\geq\frac{2n-3}{6}+2\epsilon$, $0\leq a\leq  2s+\frac{5-2n}{2}-4\epsilon$,  we have
\begin{eqnarray}
&&|I_{5}|\leq C\sum NN^{s+a-2}N_{1}^{-s}N_{2}^{-s}N_{3}^{-s}
\|P_{N_{1}}u_{1}P_{N_{2}}u_{2}P_{N_{3}}\|_{L_{xT}^{2}}\|P_{N}h\|_{L_{xT}^{2}}\nonumber\\
&&\leq C\sum N_{1}^{-\epsilon}N_{1}^{a-\frac{5-2n}{2}-2s+4\epsilon}\|P_{N}h\|_{L_{xT}^{2}}
\prod_{j=1}^{3}\|P_{N_{j}}u_{j}\|_{X_{0,b}}\nonumber\\
&&\leq C\sum N_{1}^{-\epsilon}\|P_{N}h\|_{L_{xT}^{2}}\prod_{j=1}^{3}\|P_{N_{j}}u_{j}\|_{X_{0,b}}\nonumber\\
&&\leq C\|h\|_{L_{xT}^{2}}\prod_{j=1}^{3}\|u_{j}\|_{X_{0,b}}.\label{3.086}
\end{eqnarray}

This completes the proof of Lemma 3.5.

\begin{Remark}
The restriction on $n$ arises from the fact that in \eqref{3.078}, we require $\frac{2n-3}{6}+2\epsilon\leq s\leq 1$, which forces the restriction $n\leq 4$.
\end{Remark}

\begin{Lemma}\label{Lemma3.6}
Let $ n\geq 5$, $0<T<1$, $0<\epsilon\ll1$, $s\geq\frac{2n-5}{4}+2\epsilon$, $b_{0}=-\frac{1}{2}+\frac{\epsilon}{12}$, $b=\frac{1}{2}+\frac{\epsilon}{24}$ and
\begin{align*}
0\leq a&\leq \min\{1-2\epsilon,s+\frac{5-n}{2}-4\epsilon,2s+\frac{5-2n}{2}-4\epsilon\}\\
&=\min\{1-2\epsilon,2s+\frac{5-2n}{2}-4\epsilon\}.
\end{align*}
Then, we have
\begin{eqnarray}
&&\left\|\eta\left(\frac{t}{T}\right)\mathscr{F}_{xt}^{-1}\left(\frac{|\xi|^{2}\mathscr{F}_{xt}(u^{3})}{2i\phi_{\pm}(\xi)}\right)\right\|_{X_{s+a,b_{0}}}\leq CT^{-b_{0}}\|u\|_{X_{s,b}}^{3}.\label{3.087}
\end{eqnarray}
\end{Lemma}
\noindent {\bf Proof.}
Inspired by \cite{ET2025}, we prove Lemma 3.6.
From \eqref{3.02}, to prove \eqref{3.087}, it suffices to show that
\begin{eqnarray}
&&\left\|\eta\left(\frac{t}{T}\right)|D|\langle D\rangle^{-2}(u^{3})\right\|_{X_{s+a,b_{0}}}\leq CT^{-b_{0}}\|u\|_{X_{s,b}}^{3}.\label{3.088}
\end{eqnarray}
By using \cite[Lemma 2.11]{T2006}, we have
\begin{eqnarray}
&&\left\|\eta\left(\frac{t}{T}\right)|D|\langle D\rangle^{-2}(u^{3})\right\|_{X_{s+a,b_{0}}}\leq CT^{-b_{0}}
\left\|\eta\left(\frac{t}{T}\right)|D|\langle D\rangle^{-2}(u^{3})\right\|_{X_{s+a,0}}\nonumber\\
&&=CT^{-b_{0}}\left\||D|\langle D\rangle^{-2}(u^{3})\right\|_{L_{T}^{2}H_{x}^{s+a}}.\label{3.089}
\end{eqnarray}
By using \eqref{3.089}, to prove \eqref{3.087}, it suffices to show that
\begin{eqnarray}
&&\left\||D|\langle D\rangle^{s+a-2}(u^{3})\right\|_{L_{xT}^{2}}\leq C\|u\|_{X_{s,b}}^{3}.\label{3.090}
\end{eqnarray}
According to the duality idea, to prove \eqref{3.090}, we only need to prove
\begin{eqnarray}
&&\left|\int_{0}^{T}\int_{\SR^{n}}|D|\langle D\rangle^{s+a-2}\left(\langle D\rangle^{-s} u_{1}\langle D\rangle^{-s} u_{2}\langle D\rangle^{-s} u_{3}\right)(x,t)
\overline{h}(x,t)dxdt\right|\nonumber\\
&&\leq C\|h\|_{L_{T}^{2}L_{x}}\prod_{j=1}^{3}\|u_{j}\|_{X_{0,b}}.\label{3.091}
\end{eqnarray}
We denote
\begin{eqnarray}
&&I_{6}:=
\int_{0}^{T}\int_{\SR^{n}}|D|\langle D\rangle^{s+a-2}\left(\langle D\rangle^{-s} u_{1}\langle D\rangle^{-s} u_{2}\langle D\rangle^{-s} u_{3}\right)(x,t)\overline{h}(x,t)dxdt.\label{3.092}
\end{eqnarray}
Let $P_{N}h$ and $P_{N_{j}}u_{j}(j=1,2,3)$ be the dyadic decompositions of $h$ and $u_{j}$, respectively.
Then, we have that $\supp \mathscr{F}_{x}P_{N}h\subset\{\xi\in\R^{n}:|\xi|\sim N\}$, and $\supp \mathscr{F}_{x}P_{N_{j}}u_{j}\subset \{\xi\in\R^{n}:|\xi|\sim N_{j}\}$, where $N$ and $N_{j}$ are dyadic numbers. We define $\sum=\sum\limits_{N_{1}, N_{2}, N_{3}, N}$. Without loss of generality, we assume $N_{1}\geq N_{2}\geq N_{3}$.

\noindent We consider the following cases: {\bf Case 1}: $N_{1}\leq 10$; {\bf Case
2}: $N_{1}\geq 10$, $N_{1}\gg N_{2}\geq N_{3}$ and {\bf Case 3}: $N_{1}\geq 10$, $N_{1}\sim N_{2}\gg N_{3}$; {\bf Case 4}: $N_{1}\geq 10$, $N_{1}\sim N_{2}\sim N_{3}\gg N$; {\bf Case 5}: $N_{1}\geq 10$, $N_{1}\sim N_{2}\sim N_{3}\sim N$, respectively.

\noindent {\bf Case 1}: When $N_{1}\leq 10$, by using H\"{o}lder inequality and Bernstein inequality, we have
\begin{eqnarray}
&&|I_{6}|\leq C\sum N\langle N\rangle^{s+a-2}\langle N_{1}\rangle^{-s}\langle N_{2}\rangle^{-s}\langle N_{3}\rangle^{-s}
\|P_{N_{1}}u_{1}\|_{L_{xT}^{2}}\|P_{N}h\|_{L_{xT}^{2}}\prod_{j=2}^{3}\|P_{N_{j}}u_{j}\|_{L_{xT}^{\infty}}\nonumber\\
&&\leq C\sum NN_{2}^{\frac{n}{2}}N_{3}^{\frac{n}{2}}\langle N_{1}\rangle^{-s}\langle N_{2}\rangle^{-s}
\|P_{N}h\|_{L_{xT}^{2}}\prod_{j=1}^{3}\|P_{N_{j}}u_{j}\|_{X_{0,b}}\nonumber\\
&&\leq C\sum NN_{1}^{\frac{n}{4}}N_{2}^{\frac{n}{4}}N_{3}^{\frac{n}{2}}\|P_{N}h\|_{L_{xT}^{2}}\prod_{j=1}^{3}\|P_{N_{j}}u_{j}\|_{X_{0,b}}\nonumber\\
&&\leq C\|h\|_{L_{xT}^{2}}\prod_{j=1}^{3}\|u_{j}\|_{X_{0,b}}.\label{3.093}
\end{eqnarray}
{\bf Case 2}: When $N_{1}\geq 10$,  $N_{1}\gg N_{2}\geq N_{3}$, we consider the following cases: {\bf Case 2A}: $N_{2}\leq 8$ and {\bf Case
2B}: $N_{2}\geq 8$.

\noindent {\bf Case 2A}: $N_{1}\geq 10$,  $N_{1}\gg N_{2}\geq N_{3}$, $N_{2}\leq 8$, we have $N_{1}\sim N$.
By using H\"{o}lder inequality and Bernstein inequality, $0\leq a\leq 1-2\epsilon$, we have
\begin{eqnarray}
&&|I_{6}|\leq C\sum NN^{s+a-2}N_{1}^{-s}\langle N_{2}\rangle^{-s}\langle N_{3}\rangle^{-s}
\|P_{N_{1}}u_{1}\|_{L_{xT}^{2}}\|P_{N}h\|_{L_{xT}^{2}}\prod_{j=2}^{3}\|P_{N_{j}}u_{j}\|_{L_{xT}^{\infty}}\nonumber\\
&&\leq C\sum N^{a-1}N_{2}^{\frac{n}{2}}N_{3}^{\frac{n}{2}}\langle N_{2}\rangle^{-s}\langle N_{3}\rangle^{-s}
\|P_{N}h\|_{L_{xT}^{2}}\prod_{j=1}^{3}\|P_{N_{j}}u_{j}\|_{X_{0,b}}\nonumber\\
&&\leq C\sum N_{1}^{a-1}N_{2}^{\frac{n}{2}}N_{3}^{\frac{n}{2}}\|P_{N}h\|_{L_{xT}^{2}}\prod_{j=1}^{3}\|P_{N_{j}}u_{j}\|_{X_{0,b}}\nonumber\\
&&\leq C\|h\|_{L_{xT}^{2}}\prod_{j=1}^{3}\|u_{j}\|_{X_{0,b}}.\label{3.094}
\end{eqnarray}
{\bf Case 2B}: $N_{1}\geq 10$,  $N_{1}\gg N_{2}$, $N_{2}\geq 8$,  we consider the following  cases: {\bf Case 2$B_{1}$}: $N_{3}\leq 4$ and {\bf Case
2$B_{2}$}: $N_{3}\geq 4$.

\noindent{\bf Case 2$B_{1}$}: $N_{3}\leq 4$, we have $N\sim N_{1}$.
By using H\"{o}lder inequality and \eqref{2.06}, $n\geq 5$, $s\geq\frac{2n-5}{4}+2\epsilon$, $0\leq a\leq \max\{1-2\epsilon, s+\frac{5-n}{2}-4\epsilon\}$, we have
\begin{eqnarray}
&&|I_{6}|\leq C\sum NN^{s+a-2}N_{1}^{-s}N_{2}^{-s}\langle N_{3}\rangle^{-s}
\|P_{N_{1}}u_{1}P_{N_{2}}u_{2}P_{N_{3}}u_{3}\|_{L_{xT}^{2}}\|P_{N}h\|_{L_{xT}^{2}}\nonumber\\
&&\leq C\sum N_{1}^{-\epsilon}N_{1}^{a-1+\epsilon}N_{1}^{-\frac{1}{2}+\epsilon}N_{2}^{\frac{n}{2}-1-s+2\epsilon}N_{3}^{\frac{n}{2}}
\|P_{N}h\|_{L_{xT}^{2}}\prod_{j=1}^{3}\|P_{N_{j}}u_{j}\|_{X_{0,b}}\nonumber\\
&&\leq C\sum N_{1}^{-\epsilon} N_{2}^{-s+a-\frac{5-n}{2}+4\epsilon}N_{3}^{\frac{n}{2}}\|P_{N}h\|_{L_{xT}^{2}}\prod_{j=1}^{3}\|P_{N_{j}}u_{j}\|_{X_{0,b}}\nonumber\\
&&\leq C\|h\|_{L_{xT}^{2}}\prod_{j=1}^{3}\|u_{j}\|_{X_{0,b}}.\label{3.095}
\end{eqnarray}
{\bf Case 2$B_{2}$}: $N_{3}\geq 4$, we have $N\sim N_{1}$, $N_{2}\geq N_{3}\geq4$. When $s\geq \frac{n}{2}$,
by using H\"{o}lder inequality and \eqref{2.06}, $0\leq a\leq 1-2\epsilon$, we have
\begin{eqnarray}
&&|I_{6}|\leq C\sum NN^{s+a-2}N_{1}^{-s}N_{2}^{-s}N_{3}^{-s}
\|P_{N_{1}}u_{1}P_{N_{2}}u_{2}P_{N_{3}}u_{3}\|_{L_{xT}^{2}}\|P_{N}h\|_{L_{xT}^{2}}\nonumber\\
&&\leq C\sum N_{1}^{-\epsilon}N_{1}^{a-1+\epsilon}N_{1}^{-\frac{1}{2}+\epsilon}N_{2}^{\frac{n}{2}-1+2\epsilon}N_{2}^{-s}N_{3}^{-s+\frac{n}{2}}
\|P_{N}h\|_{L_{xT}^{2}}\prod_{j=1}^{3}\|P_{N_{j}}u_{j}\|_{X_{0,b}}\nonumber\\
&&\leq C\sum N_{1}^{-\epsilon} \|P_{N}h\|_{L_{xT}^{2}}\prod_{j=1}^{3}\|P_{N_{j}}u_{j}\|_{X_{0,b}}\nonumber\\
&&\leq C\|h\|_{L_{xT}^{2}}\prod_{j=1}^{3}\|u_{j}\|_{X_{0,b}}.\label{3.096}
\end{eqnarray}
When $\frac{2n-5}{4}+2\epsilon\leq s\leq \frac{n}{2}$,
by using H\"{o}lder inequality and \eqref{2.06}, $0\leq a\leq 2s+\frac{5-2n}{2}-4\epsilon$, we have
\begin{eqnarray}
&&|I_{6}|\leq C\sum NN^{s+a-2}N_{1}^{-s}N_{2}^{-s}N_{3}^{-s}
\|P_{N_{1}}u_{1}P_{N_{2}}u_{2}P_{N_{3}}u_{3}\|_{L_{xT}^{2}}\|P_{N}h\|_{L_{xT}^{2}}\nonumber\\
&&\leq C\sum N_{1}^{-\epsilon}N_{1}^{a-1+\epsilon}N_{1}^{-\frac{1}{2}+\epsilon}N_{2}^{\frac{n}{2}-1+2\epsilon}N_{2}^{-s}N_{3}^{-s+\frac{n}{2}}
\|P_{N}h\|_{L_{xT}^{2}}\prod_{j=1}^{3}\|P_{N_{j}}u_{j}\|_{X_{0,b}}\nonumber\\
&&\leq C\sum N_{1}^{-\epsilon} N_{3}^{a-2s-\frac{5-2n}{2}+4\epsilon}\|P_{N}h\|_{L_{xT}^{2}}\prod_{j=1}^{3}\|P_{N_{j}}u_{j}\|_{X_{0,b}}\nonumber\\
&&\leq C\|h\|_{L_{xT}^{2}}\prod_{j=1}^{3}\|u_{j}\|_{X_{0,b}}.\label{3.097}
\end{eqnarray}
{\bf Case 3}: $N_{1}\geq 10$, $N_{1}\sim N_{2}\gg N_{3}$, we consider the following cases: {\bf Case 3A}: $N_{1}\gg N$ and {\bf Case
3B}: $N_{1}\sim N$.

\noindent {\bf Case 3A}: $N_{1}\geq 10$, $N_{1}\sim N_{2}\gg N_{3}$, $N_{1}\gg N$, we consider the following cases: {\bf Case 3$A_{1}$}: $N_{3}\leq 4$ and {\bf Case
3$A_{2}$}: $N_{3}\geq 4$.

\noindent{\bf Case 3$A_{1}$}: $N_{3}\leq 4$, when $N\leq 4$, by using H\"{o}lder inequality and \eqref{2.06}, $n\geq 5$, $s\geq\frac{2n-5}{4}+2\epsilon$, we have
\begin{eqnarray}
&&|I_{6}|\leq C\sum N\langle N\rangle^{s+a-2}N_{1}^{-s}N_{2}^{-s}\langle N_{3}\rangle^{-s}
\|P_{N_{1}}u_{1}P_{N_{2}}u_{2}P_{N_{3}}u_{3}\|_{L_{xT}^{2}}\|P_{N}h\|_{L_{xT}^{2}}\nonumber\\
&&\leq C\sum NN_{1}^{-\epsilon}N_{1}^{\frac{n-3}{2}-2s+4\epsilon}N_{3}^{\frac{n}{2}}
\|P_{N}h\|_{L_{xT}^{2}}\prod_{j=1}^{3}\|P_{N_{j}}u_{j}\|_{X_{0,b}}\nonumber\\
&&\leq C\sum NN_{1}^{-\epsilon}N_{3}^{\frac{n}{2}}\|P_{N}h\|_{L_{xT}^{2}}\prod_{j=1}^{3}\|P_{N_{j}}u_{j}\|_{X_{0,b}}\nonumber\\
&&\leq C\|h\|_{L_{xT}^{2}}\prod_{j=1}^{3}\|u_{j}\|_{X_{0,b}}.\label{3.098}
\end{eqnarray}
When $N\geq 4$, $n\geq 5$, $s\geq\frac{2n-5}{4}+2\epsilon$, by using H\"{o}lder inequality and \eqref{2.06}, $0\leq a\leq 1-2\epsilon$,  we have
\begin{eqnarray}
&&|I_{6}|\leq C\sum NN^{s+a-2}N_{1}^{-s}N_{2}^{-s}\langle N_{3}\rangle^{-s}
\|P_{N_{1}}u_{1}P_{N_{2}}u_{2}P_{N_{3}}u_{3}\|_{L_{xT}^{2}}\|P_{N}h\|_{L_{xT}^{2}}\nonumber\\
&&\leq C\sum N_{1}^{-\epsilon}N^{s+a-1}N_{1}^{\frac{n-3}{2}-2s+4\epsilon}N_{3}^{\frac{n}{2}}\|P_{N}h\|_{L_{xT}^{2}}
\prod_{j=1}^{3}\|P_{N_{j}}u_{j}\|_{X_{0,b}}\nonumber\\
&&\leq C\sum N_{1}^{-\epsilon}N_{1}^{\frac{n-3}{2}-s+4\epsilon}N_{3}^{\frac{n}{2}}\|P_{N}h\|_{L_{xT}^{2}}
\prod_{j=1}^{3}\|P_{N_{j}}u_{j}\|_{X_{0,b}}\nonumber\\
&&\leq C\sum N_{1}^{-\epsilon}N_{3}^{\frac{n}{2}}\|P_{N}h\|_{L_{xT}^{2}}\prod_{j=1}^{3}\|P_{N_{j}}u_{j}\|_{X_{0,b}}\nonumber\\
&&\leq C\|h\|_{L_{xT}^{2}}\prod_{j=1}^{3}\|u_{j}\|_{X_{0,b}}.\label{3.099}
\end{eqnarray}
\noindent{\bf Case 3$A_{2}$}: $N_{3}\geq 4$, when $N\leq 4$, by using H\"{o}lder inequality and \eqref{2.06}, $n\geq 5$, $s\geq\frac{2n-5}{4}+2\epsilon$, we have
\begin{eqnarray}
&&|I_{6}|\leq C\sum N\langle N\rangle^{s+a-2}N_{1}^{-s}N_{2}^{-s} N_{3}^{-s}
\|P_{N_{1}}u_{1}P_{N_{2}}u_{2}P_{N_{3}}u_{3}\|_{L_{xT}^{2}}\|P_{N}h\|_{L_{xT}^{2}}\nonumber\\
&&\leq C\sum NN_{1}^{-\epsilon}N_{1}^{\frac{n-3}{2}-2s+4\epsilon}N_{3}^{-s+\frac{n}{2}}\|P_{N}h\|_{L_{xT}^{2}}
\prod_{j=1}^{3}\|P_{N_{j}}u_{j}\|_{X_{0,b}}\nonumber\\
&&\leq C\sum NN_{1}^{-\epsilon}\|P_{N}h\|_{L_{xT}^{2}}\prod_{j=1}^{3}\|P_{N_{j}}u_{j}\|_{X_{0,b}}
\leq C\|h\|_{L_{xT}^{2}}\prod_{j=1}^{3}\|u_{j}\|_{X_{0,b}}.\label{3.0100}
\end{eqnarray}
When $N\geq 4$, $n\geq 5$, $s\geq\frac{2n-5}{4}+2\epsilon>1$, $0\leq a\leq 2s+\frac{5-2n}{2}-4\epsilon$, by using H\"{o}lder inequality and \eqref{2.06}, we have
\begin{eqnarray}
&&|I_{6}|\leq C\sum NN^{s+a-2}N_{1}^{-s}N_{2}^{-s}N_{3}^{-s}
\|P_{N_{1}}u_{1}P_{N_{2}}u_{2}P_{N_{3}}\|_{L_{xT}^{2}}\|P_{N}h\|_{L_{xT}^{2}}\nonumber\\
&&\leq C\sum N_{1}^{-\epsilon}N^{s+a-1}N_{1}^{\frac{n-3}{2}-2s+4\epsilon}N_{3}^{-s+\frac{n}{2}}\|P_{N}h\|_{L_{xT}^{2}}
\prod_{j=1}^{3}\|P_{N_{j}}u_{j}\|_{X_{0,b}}\nonumber\\
&&\leq C\sum N_{1}^{-\epsilon}\|P_{N}h\|_{L_{xT}^{2}}\prod_{j=1}^{3}\|P_{N_{j}}u_{j}\|_{X_{0,b}}\nonumber\\
&&\leq C\|h\|_{L_{xT}^{2}}\prod_{j=1}^{3}\|u_{j}\|_{X_{0,b}}.\label{3.0101}
\end{eqnarray}
{\bf Case 3B}: $N_{1}\sim N_{2}\sim N$, when $N_{3}\leq 4$,
by using H\"{o}lder inequality and \eqref{2.06}, $n\geq 5$, $s\geq\frac{2n-5}{4}+2\epsilon$, $0\leq a\leq s+\frac{5-n}{2}-4\epsilon$, we have
\begin{eqnarray}
&&|I_{6}|\leq C\sum NN^{s+a-2}N_{1}^{-s}N_{2}^{-s}\langle N_{3}\rangle^{-s}
\|P_{N_{1}}u_{1}P_{N_{2}}u_{2}P_{N_{3}}\|_{L_{xT}^{2}}\|P_{N}h\|_{L_{xT}^{2}}\nonumber\\
&&\leq C\sum N_{1}^{-\epsilon}N^{s+a-1}N_{1}^{\frac{n-3}{2}-2s+4\epsilon}N_{3}^{\frac{n}{2}}\|P_{N}h\|_{L_{xT}^{2}}
\prod_{j=1}^{3}\|P_{N_{j}}u_{j}\|_{X_{0,b}}\nonumber\\
&&\leq C\sum N_{1}^{-\epsilon}N_{1}^{a-s-\frac{5-n}{2}+4\epsilon}N_{3}^{\frac{n}{2}}\|P_{N}h\|_{L_{xT}^{2}}\prod_{j=1}^{3}\|P_{N_{j}}u_{j}\|_{X_{0,b}}\nonumber\\
&&\leq C\sum N_{1}^{-\epsilon}N_{3}^{\frac{n}{2}}\|P_{N}h\|_{L_{xT}^{2}}\prod_{j=1}^{3}\|P_{N_{j}}u_{j}\|_{X_{0,b}}\nonumber\\
&&\leq C\|h\|_{L_{xT}^{2}}\prod_{j=1}^{3}\|u_{j}\|_{X_{0,b}}.\label{3.0102}
\end{eqnarray}
When $N_{3}\geq 4$, $\frac{2n-5}{4}+2\epsilon\leq s\leq \frac{n}{2}$,
by using H\"{o}lder inequality and \eqref{2.06}, $s\geq\frac{2n-5}{4}+2\epsilon$, $0\leq a\leq  2s+\frac{5-2n}{2}-4\epsilon$, we have
\begin{eqnarray}
&&|I_{6}|\leq C\sum NN^{s+a-2}N_{1}^{-s}N_{2}^{-s}N_{3}^{-s}
\|P_{N_{1}}u_{1}P_{N_{2}}u_{2}P_{N_{3}}\|_{L_{xT}^{2}}\|P_{N}h\|_{L_{xT}^{2}}\nonumber\\
&&\leq C\sum N_{1}^{-\epsilon}N^{s+a-1}N_{1}^{\frac{n-3}{2}-2s+4\epsilon}N_{3}^{-s+\frac{n}{2}}\|P_{N}h\|_{L_{xT}^{2}}
\prod_{j=1}^{3}\|P_{N_{j}}u_{j}\|_{X_{0,b}}\nonumber\\
&&\leq C\sum N_{1}^{-\epsilon}N_{1}^{a-2s-\frac{5-2n}{2}+4\epsilon}\|P_{N}h\|_{L_{xT}^{2}}\prod_{j=1}^{3}\|P_{N_{j}}u_{j}\|_{X_{0,b}}\nonumber\\
&&\leq C\sum N_{1}^{-\epsilon}\|P_{N}h\|_{L_{xT}^{2}}\prod_{j=1}^{3}\|P_{N_{j}}u_{j}\|_{X_{0,b}}\nonumber\\
&&\leq C\|h\|_{L_{xT}^{2}}\prod_{j=1}^{3}\|u_{j}\|_{X_{0,b}}.\label{3.0103}
\end{eqnarray}
When $N_{3}\geq 4$, $s\geq \frac{n}{2}$,
by using H\"{o}lder inequality and \eqref{2.06}, $n\geq 5$, $s\geq\frac{2n-5}{4}+2\epsilon$, $0\leq a\leq  s+\frac{5-n}{2}-4\epsilon$, we have
\begin{eqnarray}
&&|I_{6}|\leq C\sum NN^{s+a-2}N_{1}^{-s}N_{2}^{-s}N_{3}^{-s}
\|P_{N_{1}}u_{1}P_{N_{2}}u_{2}P_{N_{3}}\|_{L_{xT}^{2}}\|P_{N}h\|_{L_{xT}^{2}}\nonumber\\
&&\leq C\sum N_{1}^{-\epsilon}N^{s+a-1}N_{1}^{\frac{n-3}{2}-2s+4\epsilon}N_{3}^{-s+\frac{n}{2}}\|P_{N}h\|_{L_{xT}^{2}}
\prod_{j=1}^{3}\|P_{N_{j}}u_{j}\|_{X_{0,b}}\nonumber\\
&&\leq C\sum N_{1}^{-\epsilon}N_{1}^{a-s-\frac{5-n}{2}+4\epsilon}\|P_{N}h\|_{L_{xT}^{2}}\prod_{j=1}^{3}\|P_{N_{j}}u_{j}\|_{X_{0,b}}\nonumber\\
&&\leq C\sum N_{1}^{-\epsilon}\|P_{N}h\|_{L_{xT}^{2}}\prod_{j=1}^{3}\|P_{N_{j}}u_{j}\|_{X_{0,b}}\nonumber\\
&&\leq C\|h\|_{L_{xT}^{2}}\prod_{j=1}^{3}\|u_{j}\|_{X_{0,b}}.\label{3.0104}
\end{eqnarray}
{\bf Case 4}: $N_{1}\geq 10$, $N_{1}\sim N_{2}\sim N_{3}\gg N$, when $N\leq 4$, by using H\"{o}lder inequality and \eqref{2.06}, $n\geq 5$, $s\geq\frac{2n-5}{4}+2\epsilon$, we have
\begin{eqnarray}
&&|I_{6}|\leq C\sum N\langle N\rangle^{s+a-2}N_{1}^{-s}N_{2}^{-s}N_{3}^{-s}
\|P_{N_{1}}u_{1}P_{N_{2}}u_{2}P_{N_{3}}\|_{L_{xT}^{2}}\|P_{N}h\|_{L_{xT}^{2}}\nonumber\\
&&\leq C\sum NN_{1}^{-\epsilon}N_{1}^{\frac{2n-3}{2}-3s+4\epsilon}\|P_{N}h\|_{L_{xT}^{2}}
\prod_{j=1}^{3}\|P_{N_{j}}u_{j}\|_{X_{0,b}}\nonumber\\
&&\leq C\sum NN_{1}^{-\epsilon}\|P_{N}h\|_{L_{xT}^{2}}\prod_{j=1}^{3}\|P_{N_{j}}u_{j}\|_{X_{0,b}}\nonumber\\
&&\leq C\|h\|_{L_{xT}^{2}}\prod_{j=1}^{3}\|u_{j}\|_{X_{0,b}}.\label{3.0105}
\end{eqnarray}
When $N\geq 4$, $n\geq 5$, $s\geq \frac{2n-5}{4}+2\epsilon>1$, $0\leq a\leq 2s+\frac{5-2n}{2}-4\epsilon$, by using H\"{o}lder inequality and \eqref{2.06}, we have
\begin{eqnarray}
&&|I_{6}|\leq C\sum NN^{s+a-2}N_{1}^{-s}N_{2}^{-s}N_{3}^{-s}
\|P_{N_{1}}u_{1}P_{N_{2}}u_{2}P_{N_{3}}\|_{L_{xT}^{2}}\|P_{N}h\|_{L_{xT}^{2}}\nonumber\\
&&\leq C\sum N_{1}^{-\epsilon}N^{s+a-1}N_{1}^{\frac{2n-3}{2}-3s+4\epsilon}\|P_{N}h\|_{L_{xT}^{2}}
\prod_{j=1}^{3}\|P_{N_{j}}u_{j}\|_{X_{0,b}}\nonumber\\
&&\leq C\sum N_{1}^{-\epsilon}N_{1}^{a-\frac{5-2n}{2}-2s+4\epsilon}\|P_{N}h\|_{L_{xT}^{2}}
\prod_{j=1}^{3}\|P_{N_{j}}u_{j}\|_{X_{0,b}}\nonumber\\
&&\leq C\sum N_{1}^{-\epsilon}\|P_{N}h\|_{L_{xT}^{2}}\prod_{j=1}^{3}\|P_{N_{j}}u_{j}\|_{X_{0,b}}\nonumber\\
&&\leq C\|h\|_{L_{xT}^{2}}\prod_{j=1}^{3}\|u_{j}\|_{X_{0,b}}.\label{3.0106}
\end{eqnarray}
{\bf Case 5}: $N_{1}\geq 10$, $N_{1}\sim N_{2}\sim N_{3}\sim N$, by using H\"{o}lder inequality and \eqref{2.06}, $s\geq \frac{2n-5}{4}+2\epsilon$, $0\leq a\leq  2s+\frac{5-2n}{2}-4\epsilon$,  we have
\begin{eqnarray}
&&|I_{6}|\leq C\sum NN^{s+a-2}N_{1}^{-s}N_{2}^{-s}N_{3}^{-s}
\|P_{N_{1}}u_{1}P_{N_{2}}u_{2}P_{N_{3}}\|_{L_{xT}^{2}}\|P_{N}h\|_{L_{xT}^{2}}\nonumber\\
&&\leq C\sum N_{1}^{-\epsilon}N_{1}^{a-\frac{5-2n}{2}-2s+4\epsilon}\|P_{N}h\|_{L_{xT}^{2}}
\prod_{j=1}^{3}\|P_{N_{j}}u_{j}\|_{X_{0,b}}\nonumber\\
&&\leq C\sum N_{1}^{-\epsilon}\|P_{N}h\|_{L_{xT}^{2}}\prod_{j=1}^{3}\|P_{N_{j}}u_{j}\|_{X_{0,b}}\nonumber\\
&&\leq C\|h\|_{L_{xT}^{2}}\prod_{j=1}^{3}\|u_{j}\|_{X_{0,b}}.\label{3.0107}
\end{eqnarray}

This completes the proof of Lemma 3.6.

\bigskip

\section{Proof of Theorem 1.1}
\setcounter{equation}{0}
\setcounter{Theorem}{0}

\setcounter{Lemma}{0}

\setcounter{section}{4}
This section is devoted to the proof of Theorem 1.1.

For $k=1$, by using Duhamel's formula, we have
\begin{eqnarray}
&&u(t,x)=\eta(t)\left(U_{1}(t)f(x)+U_{2}(t)g_{x}(x)\right)+\eta\left(\frac{t}{T}\right)\int_{0}^{t}U_{2}(t-s)\Delta(u^{2})(s,x)ds\nonumber\\
&&=u_{1}(t,x)+u_{2}(t,x)+u_{3}(t,x),\label{4.01}
\end{eqnarray}
where
\begin{eqnarray*}
&&u_{1}(t,x):=\eta(t)U_{1}(t)f(x),\,\, u_{2}(t,x):=\eta(t)U_{2}(t)g_{x}(x),\\
&& u_{3}(t,x):=\eta\left(\frac{t}{T}\right)\int_{0}^{t}U_{2}(t-s)\Delta(u^{2})(s,x)ds.
\end{eqnarray*}
When $s>s_{1}$, $(f,g)\in H^{s}(\R^{n})\times H^{s-2}(\R^{n})$, and $t\in[-T,T]$, $0<T<1$, by using Lemmas 2.3-2.4, 3.1-3.3 and $b=\frac{1}{2}+\frac{\epsilon}{24}$, we have
\begin{eqnarray}
&&\|u_{1}+u_{2}\|_{H^{s}}\leq C\|u_{1}+u_{2}\|_{X_{s,b}}=\left\|\eta(t)(U_{1}(t)f+U_{2}(t)g_{x})\right\|_{X_{s,b}}\nonumber\\
&&\leq  C\left(\|f\|_{H^{s}}+\|g\|_{H^{s-2}}\right),\label{4.02}\\
&&u_{3}(t,x)\in  X_{s+a,\>b}(\R^{n+1})\subset C([-T,T];H^{s+a}(\R^{n}))
\label{4.03},
\end{eqnarray}
where $s_{1}$ and $a$ are defined as in \eqref{1.04} and \eqref{1.03}, respectively.

This completes the proof of Theorem 1.1.

\bigskip

\section{Proofs of Theorems 1.2-1.3}
\setcounter{equation}{0}
\setcounter{Theorem}{0}

\setcounter{Lemma}{0}

\setcounter{section}{5}
This section is devoted to the proofs of Theorems 1.2-1.3.

We first prove Theorem 1.2.
For $k=1$, when $s>s_{2}$, $(f,g)\in H^{s}(\R^{n})\times H^{s-2}(\R^{n})$ and $t\in[-T,T]$ with $0<T<1$, by using Theorem 1.1, Lemma 2.4 and $b=\frac{1}{2}+\frac{\epsilon}{24}$, we have
\begin{eqnarray}
&&u_{3}(t,x)\in X_{\frac{n}{2}+\epsilon,\>b}(\R^{n+1})\subset C([-T,T];H^{\frac{n}{2}+\epsilon}(\R^{n}))
\label{5.01}.
\end{eqnarray}
By using \eqref{5.01}, we have
\begin{eqnarray}
&&\lim_{t\longrightarrow0}\|u_{3}(t)\|_{L_{x}^{\infty}}=0.\label{5.02}
\end{eqnarray}
From \eqref{5.02}, we have that Theorem 1.2 is valid.

Next, we prove Theorem 1.3. The proof of Theorem 1.3 requires the following lemma:
\begin{Lemma}\label{Lemma5.1}
Let $g\in H^{s}(\R^{n})(s>\frac{n}{2})$.
Then, we have
\begin{eqnarray*}
&&\lim\limits_{|x|\longrightarrow+\infty}g=0.
\end{eqnarray*}
\end{Lemma}

For Lemma 5.1, we refer to \cite[Lemma 9.2]{YYY2026}.

Now, we prove Theorem 1.3. From \eqref{5.01}, we have
\begin{eqnarray}
&&\sup_{t\in[-T,T]}\|u_{3}\|_{H^{\frac{n}{2}+\epsilon}}<\infty.\label{5.03}
\end{eqnarray}
Combining  \eqref{5.03} with Lemma 5.1, for $\forall t\in[-T,T],$ we have that
\begin{eqnarray}
&&\lim\limits_{|x|\longrightarrow+\infty}u_{3}(t,x)=0.\label{5.04}
\end{eqnarray}
From \eqref{5.04}, we have that Theorem 1.3 is valid.

This completes the proofs of Theorems 1.2-1.3.

\bigskip

\section{Proofs of Theorems 1.4-1.6}
\setcounter{equation}{0}
\setcounter{Theorem}{0}

\setcounter{Lemma}{0}

\setcounter{section}{6}

Applying Lemmas 2.3-2.4 and 3.4-3.6, together with arguments similar to those in the proofs of Theorems 1.1-1.3, we conclude that Theorems 1.4-1.6 hold.

\bigskip
\bigskip

\leftline{\large \bf Acknowledgments}
Wei Yan was supported by Natural Science Foundation of Henan province (No. 262300421062).

\end{document}